\documentclass[10pt,a4paper]{article}
\usepackage[square,sort,comma,numbers]{natbib}
\usepackage{amsthm,amssymb,amsmath}
\usepackage[pdftex]{thumbpdf}       
\usepackage{float}
\usepackage{graphicx}
\usepackage{verbatim}
\usepackage[usenames,dvipsnames]{color}
\usepackage{dsfont}
\usepackage{fullpage}
\usepackage{ifpdf}
\usepackage{comment}
\numberwithin{equation}{section}
\ifpdf
\usepackage[pdftex]{hyperref} 
\usepackage{authblk} 

\hypersetup{%
pdftitle={Heavy-traffic approximations for a layered network with limited resources},
pdfauthor={ A. Aveklouris},
pdfkeywords={Single server, SSC},
bookmarks=true,%
bookmarksnumbered=true,
pdfstartview={FitH},
pdfpagelayout={OneColumn},
colorlinks=true,
citecolor=NavyBlue,
linkcolor=RedOrange,
urlcolor=BrickRed,
bookmarksopen=true,%
bookmarksopenlevel=1,
unicode=true,%
breaklinks=true,%
}%
\else
  \newcommand\phantomsection\relax
  \newcommand{\url}[1]{#1}
  \newcommand{\href}[2]{#2}
  \usepackage{nohyperref}
\fi

\mathchardef\ordinarycolon\mathcode`\: \mathcode`\:=\string"8000 \begingroup \catcode`\:=\active \gdef:{\mathrel{\mathop\ordinarycolon}} \endgroup
 \theoremstyle{plain}              
\newtheorem{defn}{Definition}[section]
\newtheorem{rem}{Remark}[section]
\newtheorem{thm}{Theorem}[section]
\newtheorem{prop}[thm]{Proposition}
\newtheorem{lem}[thm]{Lemma}
\newcommand{\ind}[1]{\mathds{1}_{\{#1\}}}   
\newenvironment{sproof}{%
\proof}{\endproof}
\newcommand{\E}{\mathbb{E}}

\newcommand{\Cov}{\mathbb{C}}
\newcommand{\Prob}{\mathbb{P}}
\newcommand{\wt}{\widetilde}
\newcommand{\wh}{\widehat}

\newcommand{\qnp}{\mathfrak X}
\newcommand{\eps}{\epsilon}
\newcommand{\lin}{\lambda_i^n}
\newcommand{\lnn}{\lambda^n}
\newcommand{\I}{2}
\newcommand{\reg}{\varPi}
\newcommand{\f}{\varphi}

\newcommand{\p}{\bar{p}}
\newcommand{\W}{\bar{W}}
\newcommand{\x}{y}
\newcommand{\ba}{\bar}

\newcommand{\twopartdef}[4]{
	\left\{
		\begin{array}{ll}
			#1 & \mbox{if } #2 \\
			#3 & \mbox{if } #4
		\end{array}
	\right.
}
\title{Heavy-traffic approximations for a layered network with limited resources}
\author[1]{A.\ Aveklouris}
\author[1]{M.\ Vlasiou}
\author[2]{J.\ Zhang}
\author[1,3]{B.\ Zwart}
\affil[1]{Eindhoven University of Technology}
\affil[2]{Hong Kong University of Science and Technology}
\affil[3]{Centrum Wiskunde en Informatica}
\begin{document}

\maketitle
\begin{abstract}

Motivated by a web-server model, we present a queueing network consisting of two layers. The first layer incorporates the arrival of  customers at a network of two single-server nodes. We assume that the inter-arrival and the service times have general distributions. Customers are served according to their arrival order at each node and after finishing their service they can re-enter at nodes several times (as new customers) for new services. At the second layer, active servers act as jobs which are served by a single server working at speed one in a Processor-Sharing fashion. We further assume that the degree of resource sharing is limited by choice, leading to a Limited Processor-Sharing discipline. Our main result is a diffusion approximation for the process describing the number of customers in the system. Assuming a single bottleneck node and studying the system as it approaches heavy traffic, we prove a state-space collapse property. The key to derive this property is to study the model at the second layer and to prove a diffusion limit theorem, which yields an explicit approximation for the customers in the system.
\end{abstract}


\section{Introduction}\label{sec:Introduction}

We consider a network with a two-layered architecture. The first layer models the processing of customers by a network of two nodes. Each node can have multiple (but finitely many) servers. Customers are served according to their order of arrival and after finishing their service, they can re-enter at nodes several times for new services. The servers of the first layer act as jobs in the second layer, where they are simultaneously served by a common server working at speed one according to Processor-Sharing (PS) with rates depending on the number of customers in the system. Our goal is to derive an explicit approximation of the process describing the number of customers in the system.

We analyze the system as it approaches heavy traffic. Under the assumption that there is a single bottleneck, we derive explicit results for the joint distribution of the number of customers in the system by proving a diffusion limit theorem. To achieve this, we look at the system in the second layer. In this way, we can aggregate the whole system since the total workload of the system (including the future workload due to customers re-entering the queues) acts as if were that of a single server queue with two independent renewal inputs.

To derive our diffusion limit theorem, we carry out a program inspired by the work of Barmson \citep{bramson1998state} and Williams \cite{williams1998diffusion}, which consists of two main steps. First, we consider a critical fluid model, which can be thought of as a formal law of large numbers approximation under appropriate scaling. We identify the invariant states for the critical fluid model and we study the convergence to equilibrium of critical fluid model solutions as time goes to infinity. Our analysis has some similarities with the head-of-the-line processor sharing discipline as studied in \citep{bramson1998state}, but there are differences. In particular, as the degree of resource sharing at each node is finite in our case, we need to define appropriate spatial regions in which the fluid model solutions have qualitatively different behavior. Our main result is to show that a solution of the fluid model converges to equilibrium uniformly (in terms of the initial condition) on compact sets. To achieve this, we perform a time change that facilitates our analysis.

The second main step is to show a state-space collapse property for the joint queue length vector process in heavy traffic. For an appropriately defined sequence of stochastic processes, we show that the difference between this vector and an appropriate deterministic mapping of the one-dimensional total workload process vanishes. The latter process is shown to converge to a one-dimensional reflected Brownian motion.


Our work can be seen as a partial network extension of the limited processor sharing (LPS) queue, of which fluid, diffusion, and steady-state heavy traffic limit theorems have been derived \citep{zhang2009law}--\citep{zhang2008steady}. In our model, we assume that the inter-arrival and the service times have general distributions, but we consider that only one customer at each node can receive service at any time. In case that the inter-arrival and the service times are exponential, the service discipline at each station becomes irrelevant. An extension in the direction of general service times, using processor sharing at each node would require measure valued processes and is beyond the scope of the present paper. A mostly heuristic description of the results in this paper has appeared  in \cite{vlasiou2012separation}. In the classical applied probability literature, a version of our model has been investigated in a steady-state setting using boundary value techniques \citep{fayolle1982solution}; the solution in that paper may be used for numerical purposes, and is complementary to our heavy traffic limit, which yields explicit formulae, both for time-dependent as well as steady-state results.

In addition, our work is a contribution to the performance analysis of layered queueing networks. These are queueing networks where some entities in the system have a dual role (e.g., servers become customers to a higher-layer). In such systems, the dynamics in layers are correlated and the service speeds vary over time. Layered queueing networks can be characterised by separate layers (see \cite{rolia1995method} and \cite{woodside1995stochastic}) or simultaneous layers. In the first case, customers receive service with some delay. An application where layered networks with separate layers appear is the manufacturing systems e.g., \citep{dorsman2013marginal} and \citep{dorsman2015heavy}. On the other hand, in layered networks with simultaneous layers, customers receive service from the different layers simultaneously. Layered networks with simultaneous layers have applications in communications networks. An application example where layered networks with simultaneous layers (such as our model) appear naturally are web-based multi-tiered system architectures. In such environments, different applications compete for access to shared infrastructure resources, both at the software level (e.g., mutex and database locks, thread-pools) and at the hardware level (e.g., bandwidth, processing power, disk access). For background, see \cite{van2001web} and \cite{van2009dynamic}.

The paper is organized as follows. We provide a detailed model description in Section~\ref{sec: Model} and we introduce the systems dynamics. In Section~\ref{sec: Fluid analysis}, we derive the fluid model and analyze it under the assumption of a single bottleneck and heavy traffic in the network. As we see, the assumption of the single bottleneck allows us to prove a State Space Collapse (SSC) property. Then, we show that a fluid model solution converges to equilibrium uniformly on compact sets. The main result of this paper is contained in Section~\ref{Sec: Diffusion approximations}. Namely, we provide a diffusion limit theorem for the joint customer population process for this two-layered queueing network. First, we prove that the diffusion scaled total workload process converges in distribution to a reflected Brownian motion. This result together with results in \citep{bramson1998state}, lead to the main theorem..


\section{Model}\label{sec: Model}
We assume a network with two layers. In layer 1, there are $\I$ single-server nodes indexed by $i$. Customers arrive at node $i\in\{1,\I\}$ randomly one by one and have a random service requirement. A customer completing service at node $i$ may be routed at node $l$, $l\in\{1,\I\}$ for another service. It is assumed that customers are served according to their arrival order at each node;\ i.e., First In First Out. Only the first customer at each node can receive service at any time;\ i.e., the network is a Head of the Line network (HL).

In layer 2, there is a single server working at speed one. The servers of layer 1 are served by this single server simultaneously and at a rate which depends on the number of customers in the system. The model is illustrated in the following figure.
\begin{figure}[htbp]\label{fig:Model}
	\centering
	\includegraphics[scale=0.4]{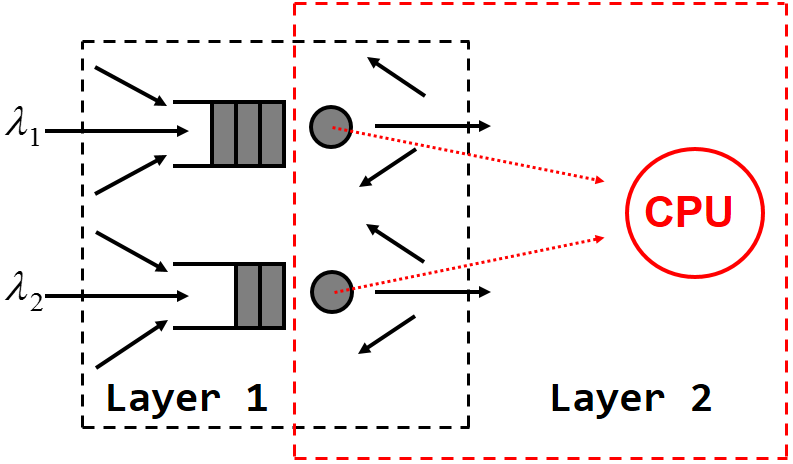}
    \caption{A two-layered network with $\I$ single-server nodes and routing.}	
\end{figure}
In Section \ref{Sec:Preliminaries and model description}, we give a formal description of the model and in Section \ref{Sec:System dynamics}, we introduce the dynamics describing the model. In the sequel, we use the subscript $i$ to refer to processes or quantities pertaining to each node and by convention, we omit the subscript to denote the 2-dimensional vector of these processes or quantities.


\subsection{Preliminaries and model description}\label{Sec:Preliminaries and model description}
In this section, we give a formal model description. Let $(\Omega,\mathcal{F},\Prob)$ be a probability space. For $T>0,$ let $\mathcal{D}[0,T]^2$ be the Skorokhod space; i.e., the space of 2-dimensional real-valued functions on $[0,T]$ that are right continuous with left limits endowed with the $J_1$ topology (as all candidate limit objects we consider are continuous, we actually only need to work with the uniform topology);  cf.\ \cite{chen2001fundamentals}. We denote by $\mathcal{B}(\mathcal{D}[0,T]^2)$ the Borel $\sigma-$algebra of $\mathcal{D}[0,T]^2$. All the processes are defined from $(\Omega,\mathcal{F},\Prob)$ to $(\mathcal{B}(\mathcal{D}[0,T]^2),\mathcal{D}[0,T]^2).$ For a process $X(t),$ we denote the uniform norm by $\|X(t)\|_T=\sup_{0 \leq t\leq T}|X(t)|$, where $|X(\cdot)|=\max_i|X_i(\cdot)|.$
We adopt the convention that all mentioned vectors are 2-dimensional columns and use $a^T$ to denote the transpose of a vector or a matrix $a$. We use $A^{-1}$ to denote the inverse of a square matrix $A$, $A^{k}$ its $k^{\text{th}}$ power, and $\|A\|$ the maximum element of $A$. Furthermore, $I$ represents the identity matrix and $e$ and $e_0$ are the vectors consisting of 1's and 0's, respectively, the dimensions of which are clear from the context. Also, $e_i$ is the vector whose $i^\text{th}$ element is 1 and the rest are all 0. Last, for a real number $x$,  its integer part is represented by $[x]$.

We start by describing the first layer. Let $u_i(j)$ for $j=2,3,\ldots$, be the time between the $(j-1)^\text{th}$ and $j^\text{th}$ external arrival at node $i$ and $u_i(1)>0$ be the residual arrival time of the first customer entering at node $i$ after time 0. We assume that the sequence $\{u_i(j)\}$ for $i=1,\I$ and $j=2,3,\ldots$ is a sequence of positive i.i.d.\ random variables with mean $1/\lambda_i$, $\lambda_i>0,$ and that $u_i(1)$ is independent of this sequence but sampled from an arbitrary distribution with the same mean.
For $i=1,\I$, define the \textit{cumulative arrival time process} $U_i(\cdot)$, as follows: $U_i(0):=0$ and $U_i(m):=\sum_{j=1}^m u_i(j)$ for $m\in \mathbb{N}$.
The number of the external arrivals at node $i$ until time $t>0$ is given by \textit{the external arrival process}
\begin{equation*}
E_i(t):= \max\{m\geq 0: U_i(m)\leq t \}.
\end{equation*}

In order to be able to define the total workload process in the system, including future service requirements due to routing, we need to introduce a sequence of random variables for any customer $j.$ For any fixed time $t\geq 0$ and $j>1$, let $v_{ii}^{(1)}(j),$ be the immediate service requirement of the $j^\text{th}$ customer (external or routed) at node $i.$ Also, we define $v_{il}^{(k+1)}(j)$ to be the service requirement of the $j^\text{th}$ customer (external or routed) at node $i$ at the $k^\text{th}$ future time he visits node $l,$ for $i,l\in\{1,2\}$ and $j,k\in\mathbb{N}.$ The sequence
$\{v_{il}^{(k)}(j)\}$, indexed by $j$, is a sequence of i.i.d.\ random variables for any fixed $i,l,k$ and for $j>1$ and has mean
$\beta_l:=1/\mu_l,$ $\mu_l>0.$ The random variable $v_{ii}^{(1)}(1)$ denotes the residual service time for the first customer being served at node $i$ at time 0;\ it is independent of the sequence $\{v_{il}^{(k)}(j)\}$ but sampled from an arbitrary distribution with the same mean.  In addition, we assume that all the above-mentioned random variables have finite second moments (more precisely, we need a Lindeberg-type condition to hold to make sure that the exogenous input processes satisfy a functional central limit theorem; see Section \ref{Sec: Diffusion approximations} for more details).
We define the \textit{cumulative service time process} as $V_i(0):=0$ and for $m\in \mathbb{N}$,
\begin{equation*}\label{cumulative service time process}
V_i(m):= \sum_{j=1}^m v_{ii}^{(1)}(j)\qquad i=1,\I,
\end{equation*}
and the counting process
\begin{equation}\label{eq:counting process}
S_i(t):= \max\{m\geq 0: V_i(m)\leq t \}.
\end{equation}
We shall use the random variables $\{v_{il}^{(k+1)}(j)\}$ to count the future workload in the system at time $t$. For a fixed $t>0$, $v_{il}^{(k+1)}(j)$ represents the $k^\text{th}$ future service requirement of the $j^\text{th}$ customer waiting to being served at node $i$ and routed at node $l.$ This event will occur after time $t$ and after the completion of his service at node $i.$

Customers can move between queues according to Markovian routing. To describe the routing process, we define the following quantities. Let $P$ be the (square) routing matrix of dimension $\I$. It is assumed that it is substochastic with a spectral radius less than one;\ i.e., its largest eigenvalue is less than one. In other words, the network is open, so the following relations hold:
\begin{center}
$(I-P^T)^{-1}=I+\sum_{k=1}^{\infty} (P^T)^k$\quad \text{and}\quad $\lim_{k\rightarrow \infty}(P^T)^k=0.$
\end{center}

For any customer $j$ at node $i$ (external or routed), we define the random variables
$\f^{(k)}_{il}(j)=1$ if the $j^\text{th}$ departing customer from node $i$ is routed to node $l$ in $k$ steps.
The probability of this event is given by
\begin{equation*}\label{eq:k-routing}
\Prob(\f^{(k)}_{il}(j)=1)=p^{(k)}_{il},
\end{equation*}
where $p^{(k)}_{il}$ denotes the $(i,l)^\text{th}$ element of the matrix $P^k$.
For $i=1,2$, we define the 2-dimensional random vector
\begin{equation*}\label{eq:vector of Phi}
\f^{(k)}_{i}(j):=
(\f^{(k)}_{i1}(j),\f^{(k)}_{i2}(j))^T.
\end{equation*}
Note that $\f^{(k)}_{i}(j)$ can take values in the set $\{e_0,e_1,e_2\}$, where $\f^{(k)}_{i}(j)=e_0$ means that the $j^\text{th}$ customer leaves the system. Let $\p^{(k)}_{i}$ be the $i^\text{th}$ column of the matrix $(P^T)^k$. The expectation and the covariance matrix of $\f^{(k)}_{i}(j),$ for $i=1,2, $ are given by
\begin{center}
$\E(\f^{(k)}_{i}(j))=\p^{(k)}_{i}$  \quad \text{and}\quad
$
\Cov(\f^{(k)}_{i}(j))=\begin{bmatrix}
p^{(k)}_{i1}(1-p^{(k)}_{i2})      & -p^{(k)}_{i1}p^{(k)}_{i2}  \\
-p^{(k)}_{i1}p^{(k)}_{i2}     & p^{(k)}_{i2}(1-p^{(k)}_{i1})
\end{bmatrix}.
$
\end{center}

Now, we can define the routing process, which counts the number of customers who are routed from node $l$ to node $i$ as
\begin{equation*}\label{eq:routing_process}
 \Phi_{li}(m):=\sum_{j=1}^{m} \f_{li}^{(1)}(j), \quad i,l=1,2,\ m\in \mathbb{N}.
\end{equation*}
The total arrival rate at node $i$, $\gamma_i$, is given by the solution of the following traffic equations
\begin{equation*}
\gamma_i=\lambda_i+\sum_{l=1}^{\I}p_{li}^{(1)}\gamma_l,\ i=1,\I.
\end{equation*}
In vector form, this can be written as
\begin{equation}\label{eq:trafficVector}
 \gamma=(I-P^T)^{-1}\lambda.
\end{equation}
It is shown in \cite[Theorem 7.3]{chen2001fundamentals}
that under the assumptions described above,
\eqref{eq:trafficVector} has a unique solution $\gamma=(\gamma_1,\gamma_2)^T$. The traffic intensity of node $i$ is $\rho_i:=\gamma_i / \mu_i.$

Now, we describe the service discipline at the second layer. Here, there is a single server. The servers of layer 1 become jobs at layer 2 in the sense they are served by the server of layer 2 simultaneously and at a rate that depends on the number of customers in layer 1 (at any time). The rate that each node receives is given by the service allocation function
$R(\cdot):\mathbb{R}_+^2\rightarrow \mathbb{R}_+^2$, with $R(\cdot):=(R_1(\cdot),R_2(\cdot))^T$ and for $i=1,\I,$
\begin{equation}\label{eq:ServiceAllocation}
R_i(q):= \twopartdef { \frac{\min\{ {q_i},{K_i}\}}{\sum_{j=1}^{\I}\min\{ {q_j},{K_j}\}}} {q_i\neq0,} {0} {q_i=0.}
\end{equation}
The quantity $q_i$ represents the number of customers at node $i.$
The $\I$-dimensional vector $K=(K_1,K_\I)^T$ is constant and we call it \textit{the degree of resource sharing.}  We assume that it is always finite and the user can choose it as a parameter of the system.
Observe that $K$ in this service allocation function guarantees a minimum service rate for each customer in the system. Also, note that the above function is Lipschitz continuous for $q\neq (0,0).$

We make the additional assumption that there exists a unique bottleneck in our system, which w.l.o.g.\ we let it be node 1. The definition of bottleneck is given below.
\begin{defn}[Bottleneck]\label{def:Bottleneck}
Node $i$ is a bottleneck if $i=\arg \min_j\frac{\mu_jK_j}{\gamma_j},\ j=1,\I.$
\end{defn}
\noindent
By the previous definition, a straightforward inequality  follows
\begin{equation}\label{In:Bot.AS}
\frac{\rho_1}{K_1}>\frac{\rho_2}{K_2}.
\end{equation}
Observe that, if $K_1=K_2,$ an intuitive explanation of the above definition is that the average occupancy of the server at node 1 is strictly greater than the server at node 2. In case of multi-server nodes where $K_i$ represents the number of servers an node $i,$ the fraction $\frac{\rho_i}{K_i}$ is the average occupancy of a server at node $i.$


\subsection{System dynamics}\label{Sec:System dynamics}

In this section, we introduce the dynamics that describe our model. We denote by $Q_i(t)$ the number of customers at node $i$ at time $t$. This is given by
\begin{equation}\label{a}
Q_i(t)=Q_i(0)+E_i( t)+\sum_{l=1}^{\I}\Phi_{li}\Big(S_l\big(T_l(t)\big)\Big)-S_i\big(T_i(t)\big),
\end{equation}
where $Q_i(0)$ denotes the number of customers initially at node $i.$ We define the \textit{cumulative service time} of the server at node $i$ as
\begin{equation}\label{c}
T_i(t)=\int_{0}^{t} R_i(Q(s))ds.
\end{equation}
This quantity can be viewed as the effort that the server of node $i$ has put
in processing customers during $[0,t].$
Note that as the allocation function might be less than one, the above process is not necessarily equal to the amount of time that the server at node $i$ is busy during $[0,t].$ In case the other node is empty during $[0,t],$  \eqref{c} coincides with the busy time at node $i$. Recall that $E_i(t)$ is the number of external arrivals at node $i$ up to time $t.$
Observe that $S_i\big(T_i(t)\big),$ which is a composition of the renewal process \eqref{eq:counting process} and the process $T_i(t)$, represents the number of departures at node $i$ until time $t$.
Furthermore, the \textit{total arrival process} is given by
\begin{equation}\label{TotalArPr}
A_i(t)=E_i(t)+\sum_{l=1}^{\I}\Phi_{li}\Big(S_l(T_l(t))\Big).
\end{equation}
The amount of time that both servers at the nodes are idle during $[0,t]$  is given by the 1-dimensional process
\begin{equation}\label{d}
Y_{L_2}(t)=t-\sum_{i=1}^{\I} T_i(t).
\end{equation}
Alternatively, we can see this quantity as the idle time of the server in layer 2 during $[0,t]$. Further, the \textit{immediate workload} at node $i$ at time $t$ is defined as
\begin{equation}\label{b}
W_i(t)=V_i\Big(Q_i(0)+A_i(t)\Big)-T_i(t).
\end{equation}
Observe that $W_i(t)$ is nonnegative for any $t\geq 0$. Last, due to the work-conserving property in layer 2, the following relation holds:
\begin{equation}\label{e}
Y_{L_2}(t)\textrm{ increases }\Rightarrow W_1(t)+W_2(t)=0, \qquad t\geq0,\ i=1,\I.
\end{equation}
Recall that when we omit the subscript $i$, we refer to the 2-dimensional column vector of the corresponding process/quantity; for example $A(\cdot)=(A_1(\cdot),A_2(\cdot))^T$ and
$W(\cdot)=(W_1(\cdot),W_2(\cdot))^T.$
All the essential information of the evolution of the system is contained is the following 6-tuple
$$\qnp(\cdot):=(A(\cdot),S(\cdot),Q(\cdot),T(\cdot),Y_{L_2}(\cdot),W(\cdot)).$$

In addition, the total (immediate and future) workload of the system plays a key role in our analysis. First, we define the remaining service requirement of the $j^{\text{th}}$ customer waiting to be served at node $i=1,\I$ as
\begin{equation}\label{eq:total S requirement}
s_i(j):=v_{ii}^{(1)}(j)+s'_i(j),
\end{equation}
where
\begin{equation*}\label{eq:Future requirement}
s'_i(j):=\sum_{l=1}^{\I}\sum_{k=1}^{\infty}
\f^{(k)}_{il}(j)v_{il}^{(k+1)}(j)
\end{equation*}
is the future service requirement of the above-mentioned customer. Observe that, for an external arrival, $s_i(j)$ is the total service requirement.
The first and the second moments of \eqref{eq:total S requirement} are given (in vector form) by
\begin{equation}\label{eq:Total time}
 \tau:=\E(s(j))=(I-P)^{-1}\beta
\end{equation}
and
\begin{equation}\label{eq:secont moment of Total time}
 \tau^{(2)}:=\E(s^2(j))=(I-P)^{-1}( \E(v^2(j))+2\beta(P\tau)).
\end{equation}

Now, we can define the (1-dimensional) total workload of the system as
\begin{equation}\label{eq:total_workload}
W_{Tot}(t):=\sum_{i=1}^{2} W_i(t)+
\sum_{i=1}^{\I}\sum_{j=S_i\big(T_i(t)\big)+1}^{Q_i(0)+A_i( t)}
s'_i(j).
\end{equation}
In case that $S_i\big(T_i(t)\big)=Q_i(0)+A_i( t)$, i.e., there are no customers at node $i$, we understand the second sum of the last equation as zero. Obviously, the total workload is not a Markov process as it is dependent on future service requirements. In Section~\ref{Convergence of total workload}, we shall see that under an appropriate scaling (i.e., the diffusion scaling) the dependence of the total workload on the future vanishes.

Last, as our network is HL, only one customer can be in service at node $i$ at any time. This property gives an upper and a lower bound for the cumulative service time \eqref{c} at node $i,$ which is given in \cite[Inequality~2.13]{bramson1998state}; namely
\begin{equation}\label{eq:HL}
	V_i(S_i(T_i(t)))\leq T_i(t)< V_i(S_i(T_i(t))+1).
\end{equation}

We have so far defined the system dynamics for the above-mentioned two layered network and stated all the assumptions we need for our analysis.
We are now ready to study the fluid model of this network, which is the first essential step to show a SSC property.


\section{Fluid analysis}\label{sec: Fluid analysis}
In this section, we study a critical fluid model, which is a deterministic model and can be thought of as a formal law of large numbers approximation under appropriate scaling. We shall give a rigorous proof of the last statement in the next section.

The main goal is to prove uniform convergence (w.r.t. the initial condition) on compact sets for the fluid model under the critical loading assumption; i.e., the traffic intensity of the network is one. First, we find the invariant points (or equilibrium states) and define an appropriate lifting map
which describes these points. Then, we define a time-changed version of the original fluid model and we show that it is enough to prove the convergence for the time-changed function. As the time-changed function is given by a piece-wise linear ODE, we are able to find the solution and to show the convergence. Because the degree of resource sharing (the vector $K$) is finite we need to separate the state space in suitable regions and to distinguish cases depending on initial conditions.

\subsection{Definition and invariant points}\label{Subsec: Invariant points}
The traffic intensity of the network is given by $\rho:=\beta^T \gamma=\sum_{i=1}^{\I}\rho_i.$ We make the critical loading assumption, i.e.,
\begin{equation}\label{eq:critical loading}
\rho=\rho_1+\rho_2=1.
\end{equation}
To derive the fluid model equations we replace any random quantity in \eqref{a}$\textup{--}${\eqref{e} with its mean. The fluid model equations are given by
\begin{eqnarray}
&&\ba{Q}_i(t)=\ba{Q}_i(0)+ \lambda_i t +\sum_{l=1}^{\I}p_{li}\mu_l \ba{T}_l(t)-\mu_i \ba{T}_i(t),\label{eq:Fluid queue}\\
&&\ba{T}_i(t)=\int_{0}^{t} R_i(\ba{Q}(s))ds,\label{eq:Fluid busy time}\\
&&\sum_{i=1}^{\I} \ba{T}_i(t)+\ba{Y}_{L_2}(t)=t,\label{work concerving}\\
&&\ba{W}_i(t)=\beta_i\Big(\lambda_it+\sum_{l=1}^{\I}p_{li}\mu_l \ba{T}_l(t)+\ba{Q}_i(0)\Big)-\ba{T}_i(t), \label{eq:Fluid Im work}\\
&&\ba{Y}_{L_2}(t)\textrm{ increases }\Rightarrow \ba{W}_1(t)+\ba{W}_2(t)=0, \qquad t\geq0,\ i=1,\I.\label{eq:Fluid idle time}
\end{eqnarray}
We can show that the immediate workload in the fluid model can be written as $\ba{W}_i(t)=\beta_i \ba{Q}_i(t).$
\begin{defn}[Fluid model]
We say that a $2$-dimensional vector $\ba{Q}(\cdot)$ with non-negative components is a solution of the fluid model if it is continuous and satisfies \eqref{eq:Fluid queue}$\textup{--}$\eqref{eq:Fluid idle time} for
$t\in [0,\delta)$, and $Q(t)=0$ for $t\geq \delta$, with $\delta = \inf \{ t: Q(t)=0\}$.
\end{defn}
We define an auxiliary quantity, which can be interpreted as the total workload in the fluid model. It is defined by function $\ba{Q}(\cdot)$ as follows:
\begin{equation*}\label{eq:Fluid_ap_W}
\begin{split}
\W_{Tot}(t)=\beta^T\ba{Q}(t)+\sum_{k=1}^{\infty}\sum_{i=1}^{2}\beta^T\p^{(k)}_{i}\ba Q_i(t)
&=\beta^T\ba{Q}(t)+\sum_{k=1}^{\infty}\beta^T(P^T)^k\ba Q(t)\\
&=\beta^T(I-P^T)^{-1}\ba Q(t)=\tau^T \ba Q(t).
\end{split}
\end{equation*}
A useful result in our analysis is that the fluid total workload in the system remains constant under the critical loading assumption.
\begin{prop}\label{workload is constant}
For any fluid model solution $\ba{Q}(\cdot)$, we have that
\begin{equation*}\label{eq:W=constant}
\W_{Tot}(t)=
\beta^T(I-P^T)^{-1}\ba {Q}(t)
=\W_{Tot}(0).
\end{equation*}
\end{prop}
\begin{proof}
If $\ba{Q}(0)=0,$ then $\W_{Tot}(t)\equiv0,$ for $t\geq 0.$ Let $\ba{Q}(0)>0.$ Assume now $\ba{Q}(0)=0$. By definition, $\ba{Q}(t)$ is continuous, so let $t$ be such that $\ba{Q}(t)$ is positive in a neighborhood of $t$. Calculating the derivative of the total workload at time $t$, we derive
\begin{equation*}
\begin{split}
\W_{Tot}'(t)=&\beta^T(I-P^T)^{-1}\ba {Q}'(t)
=\beta^T(I-P^T)^{-1}(\lambda-(I-P^T)(\mu \circ R(\ba{Q}(t))\\
=&\beta^T \gamma-\beta^T (\mu \circ R(\ba{Q}(t)))
=\beta^T \gamma-\sum_{i=1}^{2} R_i(\ba{Q}(t))=0.
\end{split}
\end{equation*}
The last equation holds due to \eqref{eq:critical loading} and the property of the service allocation function; i.e.\ $\sum_{i=1}^{2} R_i(\ba{Q}(t))=1.$ It follows that $\W_{Tot}(t)=\W_{Tot}(0)$ for $t\geq 0.$
Thus, $\W_{Tot}(t)$ is constant on $[0,\delta]$. Combining this with the continuity of $\ba{Q}(t)$, we see that, necessarily, $\delta=\infty$. Thus, the result extends to all positive $t$.
\end{proof}
%
In the following lemma, we show that there exists a solution to the fluid model equations for all non-zero initial states and it is unique.
\begin{lem}[Existence and uniqueness]\label{Existence and uniqueness}
For any $\ba{Q}(0)\in \mathbb{R}_+^2 \backslash\{0\}$ there exists a unique solution to the fluid model equations.
\end{lem}
\begin{proof}
Let $\ba{Q}(0)>0$. A by-product of the previous lemma is that $\delta=\infty$. Thus, we can discard the origin and define the function
$\Psi(\cdot):\mathbb{R}_+^2\backslash\{0\}\rightarrow \mathbb{R}_+^2$
as
\begin{equation}\label{eq:psi function}
\Psi(\cdot):=\lambda-(I-P^T)(\mu \circ R(\cdot)),
\end{equation}
where $\mu \circ R(\cdot)$ indicates the Hadamard product;\ i.e., $\mu \circ R(\cdot)=(\mu_1 R_1(\cdot),\mu_2 R_2(\cdot))^T$. This function is Lipschitz continuous because $R(\cdot)$ is. Now, note that \eqref{eq:Fluid queue} can be written as
\begin{equation}\label{eq:ODE}
\ba{Q}'(t)=\Psi(\ba{Q}(t)),\quad t\geq 0,
\end{equation}
where the prime denotes the derivative with respect to time. The existence and uniqueness follows directly by
\citep[Section 10, Theorem IV]{WalterODE}.
\end{proof}
\noindent

Now, we characterize the invariant points $x\in\mathbb{R}^\I_+$ of the fluid model. Before we state our result, we first proceed in an informal manner.  Equate the total rate into node $i$ with the total rate out of node $i$. That is,
\begin{equation}\label{ie}
\gamma_i=\mu_i R_i(x).
\end{equation}
Thus, we have that for the points on the invariant manifold (i.e.\ the set of the invariant points), $\rho_i= R_i(x)$. Using the definition of a bottleneck and keeping in mind that we assume node 1 to be the bottleneck, we now describe the invariant points.
We know by \eqref{In:Bot.AS} that $\frac{\rho_1}{K_1}>\frac{\rho_2}{K_2}$,
which yields
\begin{equation*}
\frac{K_1}{K_2}<\frac{\rho_1}{\rho_2}=\frac{R_1(x)}{R_2(x)}.
\end{equation*}
Thus, by combining the last inequality and the definition of the service allocation function \eqref{eq:ServiceAllocation}, we have that the following inequality holds
\begin{equation*}
\min\{ {x_2},{K_2}\}<\frac{K_2}{K_1}\min\{ {x_1},{K_1}\}\leq K_2.
\end{equation*}
The last inequality implies that for all invariant points $x=(x_1,x_2)$ of the fluid model, we have that
$
x_2<K_2.
$
Thus, solving \eqref{ie} for $x_2$ now yields
$
x_2=\frac{\rho_2}{\rho_1}\min\{ {x_1},{K_1}\}.
$
The invariant manifold is thus given by
\begin{equation}\label{eq:IM}
\mathcal{I}=\bigg\{x\in \mathbb{R}^\I_+: x_2=\frac{\mu_1}{\gamma_1}\frac{\gamma_2}{\mu_2}\min\{ {x_1},{K_1}\}\bigg\}.
\end{equation}
We make the previous arguments rigorous by showing that a sufficient and necessary condition of the fluid queue length to remain constant in time is the initial state lies on the invariant manifold.
\begin{prop}\label{starting on IM}
Let $\ba{Q}(t)$ be a solution of \eqref{eq:ODE}. Then, $\ba{Q}(t)=\ba{Q}(0)$ for all $t\geq 0$ if and only if $\ba{Q}(0)\in \mathcal{I}.$
\end{prop}
\begin{proof}
The definition of the invariant manifold of an ODE is the set of all initial states such that the function remains constant; i.e., $\{\ba{Q}(0)\in \mathbb{R}^\I_+: \ba{Q}(t)=\ba{Q}(0), t\geq 0 \}$. Suppose now that $\ba{Q}(0)\in \mathcal{I}$. In this case, by the definition of an invariant point, $\ba{Q}(t)$ should be constant. So, $\ba{Q}(t)=\ba{Q}(0)\in \mathcal{I}.$

Now, supposing that $\ba{Q}(t)=\ba{Q}(0)\in \mathbb{R}_+^2$ for all $t\geq 0,$ then it follows from  the previous discussion that
$\ba{Q}(0)\in \mathcal{I}$.

\end{proof}

Having found the invariant (or equilibrium) points of the fluid model, we now turn to its stability property, namely the convergence of the solutions of fluid model equations to the invariant manifold as time goes to infinity.

\subsection{Convergence to the invariant manifold for the fluid model}
Let $x^\ast$ be the critical point in the invariant manifold where $x_1^\ast=K_1$, which means that $x^\ast_2=\frac{\mu_1}{\gamma_1}\frac{\gamma_2}{\mu_2}K_1$.
For this point, we define the critical workload as (cf.\ \eqref{workload is constant})
\begin{equation}\label{eq:critical}
w^\ast:= \beta^T (I-P^T)^{-1}x^\ast=\tau^T x^\ast.
\end{equation}
In order to prove a SSC property based on the critical workload level $w^\ast,$ we define a lifting map, $\Delta:\mathbb{R}_+\rightarrow \mathbb{R}_+^\I,$ as follows:
\begin{equation}\label{d2}
\begin{aligned}
\Delta_1(w)&:=\frac{\min\{w,w^\ast\}}{w^\ast}K_1+\frac{\max\{w-w^\ast,0\}}{\tau_1}, \\
\Delta_2(w)&:=\frac{\min\{w,w^\ast\}}{w^\ast} \frac{\rho_2 K_1}{\rho_1}.
\end{aligned}
\end{equation}
Note that the lifting map is Lipschitz continuous with constant
$C_1=\max\{2\frac{K_1}{w^\ast}+\mu_1, 2C_2\}$ where $C_2= \frac{\mu_1 K_1}{\lambda_1 w^\ast} \max_i \frac{\lambda_i}{\mu_i}$. In the sequel, we show that fluid model solution converges to the invariant manifold as $t$ goes to infinity.
%
\begin{thm}[Convergence to the invariant manifold for the fluid model]\label{Thm: Convergence of the fluid model}
If $\ba{Q}(0)=(\ba{Q}_1(0),\ba{Q}_2(0))\in[0,M]^2$ for some $M>0,$ then for any $\epsilon >0,$ there exists a $t_0 \geq 0$ (independent of $M$) such that
\begin{equation}\label{eq:Eq}
\sup_{\ba{Q}(0)\in[0,M]^2}|\ba{Q}(t)-\Delta\W_{Tot}(0)|\leq \epsilon,
\end{equation}
for $t>t_0$ and $\Delta\W_{Tot}(0)$ is an invariant state.
\end{thm}
\begin{sproof}
Here, we give a sketch of the proof. The complete proof is extended in the rest of this section. The first step is to define a function $y(\cdot)$ and a function $G(\cdot),$ and to show that $y(\cdot)$ can be interpreted as a time-change of $\ba{Q}(t),$ namely $\ba{Q}(t)=y(G(t)).$ Then, we show that the convergence of the time changed version implies the convergence of the original function $\ba{Q}(\cdot).$
To this end,
let $\Xi=\{\xi_{ij}\}$, $i,j \in \{1,2\}$ be the matrix
\begin{equation*}
\Xi=\begin{bmatrix}
\lambda_1+\mu_1p_{11}-\mu_1&\lambda_1+\mu_2p_{21}\\
\lambda_2+\mu_1p_{12}&\lambda_2+\mu_2p_{22}-\mu_2
\end{bmatrix}.
\end{equation*}
Define a function $y(\cdot):[0,\infty) \rightarrow [0,\infty)^2$ such that
$y_i(0)=\ba{Q}_i(0)$ and
\begin{equation}\label{eq:System ODE}
y'(t) = \Xi \Big( \min\{y_1(t),K_1\}, \min\{y_2(t),K_2\} \Big)^T .
\end{equation}
We shall show that the above-defined function can be interpreted as a time-change of $\ba{Q}(t).$
Let $G(\cdot):[0,\infty) \rightarrow [0,\infty)$ be the solution of
\begin{equation*}
G'(t)=\frac{1}{\sum_{i=1}^{\I} \min\{y_i(G(t)),K_i\}}.
\end{equation*}
Note that $G(\cdot)$ is continuous and that
\begin{equation*}
G(t)=\int_{0}^{t}
\frac{1}{\sum_{i=1}^{\I} \min\{y_i(G(s)),K_i\}} ds
\geq \frac{1}{\sum_{i=1}^{\I}  K_i}t.
\end{equation*}
This means that the function $G(\cdot)$ is strictly increasing and unbounded in time, which implies that $G(\cdot)$ is also invertible.
The original function $\ba{Q}(t)$ can be interpreted as
$\ba{Q}_i(t)=y_i(G(t)),$ for $t>0.$
To see this,
\begin{equation*}
\begin{split}
\ba{Q}'(t)=y'(G(t))G'(t)
&=\Xi \Big(\min\{y_1(G(t)),K_1\}, \min\{y_2(G(t)),K_2\}\Big)^T
\frac{1}{\sum_{i=1}^{\I} \min\{y_i(G(t)),K_i\}}\\
&=\Xi \Big(R_1(\ba{Q}(t)), R_2(\ba{Q}(t))\Big)^T=\Psi(\ba{Q}(t)),
\end{split}
\end{equation*}
where the function $\Psi(\cdot)$ is defined in \eqref{eq:psi function}.

The idea now is to prove \eqref{eq:Eq} by showing that $\x(t)$ converges as $t\rightarrow \infty.$ We will formally do so in the remainder of Section~\ref{sec: Fluid analysis}. Assuming that $\x(t)$ converges, we now show that $\ba {Q}(t)$ converges. To see this, if for any $\epsilon>0$ there exists $t_0$ (independent of $M$) such that
\begin{equation}\label{eq:Convergence of time-ch}
\sup_{y(0)\in[0,M]^2}|y(t)-\Delta\W_{Tot}(0)|
\leq \epsilon, {\hskip 25pt} t>t_0
\end{equation}
then
\begin{equation*}
\begin{split}
\sup_{\ba{Q}(0)\in[0,M]^2}|\bar{Q}(t)-\Delta\W_{Tot}(0)|
=&\sup_{y(0)\in[0,M]^2}|y(G(t))-\Delta\W_{Tot}(0)|
=\sup_{y(0)\in[0,M]^2}|y(u)-\Delta\W_{Tot}(0)|,
\end{split}
\end{equation*}
where $u=G(t).$
For $u>G^{-1}(t_0),$ the last term becomes smaller that $\epsilon.$
\end{sproof}

The remainder of the current section is devoted to the proof of \eqref{eq:Convergence of time-ch}. To do it, we first define appropriate spatial regions in which the fluid model solutions have qualitatively different behavior and we solve \eqref{eq:System ODE} in these regions. This is done in Section~\ref{Sub:Solution of time-changed ODE}. In Section~\ref{Local analysis: establishing monotonicity},  these solutions are used to show that both $y_1(t)$ and $y_2(t)$ are monotone in $t$. A crucial observation is that we only need to look at $y_2(t)$, as $\tau_1 y_1(t)+\tau_2y_2(t)=  \bar{W}_{Tot}(0)$  by Proposition~\ref{workload is constant}. This paves the way for a global convergence analysis of $y(t)$, also establishing the desired uniformity. This is done in Section~\ref{Convergence to Invariant Manifold}.


\subsection{Explicit local solutions of time-changed ODE}\label{Sub:Solution of time-changed ODE}
Because we assume that the degree of resource sharing $K$ is finite, the form of \eqref{eq:System ODE} depends on the value of $\x(t).$ For this reason, we need to define the following regions. For
$x=(x_1,x_2)\in\mathbb{R}_+^2,$ we define
\begin{align*}
&\reg_1=\{(x_1,x_2)\in \mathbb{R}^2: x_1\leq K_1, x_2\geq K_2 \},\qquad
\reg_2=\{(x_1,x_2)\in \mathbb{R}^2: x_1\geq K_1, x_2\geq K_2 \},\\
&\reg_3=\{(x_1,x_2)\in \mathbb{R}^2: x_1\geq K_1, x_2\leq K_2\},\qquad
\reg_4=\{(x_1,x_2)\in \mathbb{R}^2: x_1\leq K_1, x_2\leq K_2 \}.
\end{align*}
The following picture shows these regions and the invariant manifold as defined in \eqref{eq:IM}.
\begin{figure}[H]\label{fig:Regions}
\centering
	\includegraphics[scale=0.3]{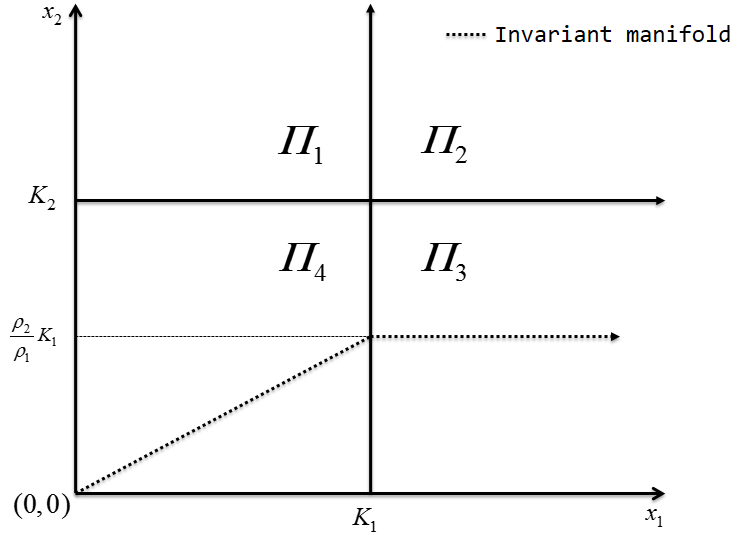}
	\caption{The regions $\reg_i$ and the invariant manifold.}	
\end{figure}

Let $\x(0)=(\x_1(0),\x_2(0))^T\in\mathbb{R}_+^2.$ We solve the time-changed ODE, which is given by \eqref{eq:System ODE}, in regions $\reg_3$ and $\reg_4$ (considering these two regions only is sufficient for our purposes).
It is useful to observe the relations between the coefficients of matrix $\Xi$ in \eqref{eq:System ODE}, which we use later.
We know by the definition of the total arrival rate that
\begin{equation}\label{eq:arrival rate 2}
\rho_2(1-p_{22})=\frac{\lambda_2}{\mu_2}+\frac{p_{12}\gamma_1}{\mu_2}.
\end{equation}
The constant $\xi_{22}$ can be expressed as
\begin{equation}\label{eq:relation1}
\begin{split}
\xi_{22}=\lambda_2+\mu_2(p_{22}-1)
\overset{\eqref{eq:arrival rate 2}}=
\lambda_2-\frac{1}{\rho_2}(\lambda_2+p_{12}\gamma_1)
=-\frac{\rho_1}{\rho_2} \xi_{21}<0.
\end{split}
\end{equation}
In a similar way, we can obtain
\begin{equation}\label{eq:relation2}
\xi_{11}=-\frac{\rho_2}{\rho_1} \xi_{12}<0.
\end{equation}
By definition \eqref{eq:Total time}, $\tau_2$ can be written as $\tau_2(1-p_{22})=\beta_2+p_{21}\tau_1.$
Also, by \eqref{eq:trafficVector} and \eqref{eq:Total time}, we have that
$\tau^T\lambda=\beta^T(I-P^T)^{-1}\lambda=\beta^T\gamma=1.$
Combining the previous two equations, we get
\begin{equation*}\label{eq:relation3}
\xi_{12}=-\frac{\tau_2}{\tau_1} \xi_{22}.
\end{equation*}
Last, by \eqref{eq:relation1} and \eqref{eq:relation2} we have that
\begin{equation*}\label{eq:relation4}
\xi_{21}=-\frac{\tau_1}{\tau_2} \xi_{11}.
\end{equation*}


\subsubsection{Solution in region $\reg_3$}\label{Subsec:Solution in region 3}
Assuming that $y(s)$ is in region $\reg_3$ for $s\in [0,t]$, we can directly solve the second equation of system \eqref{eq:System ODE} since it is independent of $y_1(t)$:
\begin{equation*}
y_2'(t)=\xi_{21}K_1+\xi_{22}y_2(t).
\end{equation*}
Then, by using \eqref{eq:relation1}, the solution is given by
\begin{equation*}\label{explicit solution y'2 in region 3}
y_2(t)=(y_2(0)-\frac{\rho_2}{\rho_1}K_1)
\exp\{\xi_{22}t\}+
\frac{\rho_2}{\rho_1}K_1.
\end{equation*}
Now, we can easily obtain the solution of the first equation of system \eqref{eq:System ODE}. By the relations between the coefficients of \eqref{eq:System ODE}, this solution is given by
\begin{equation}\label{eq:explicit solution y'1 in region 3}
y_1(t)=-\frac{\tau_2}{\tau_1}
(y_2(0)-\frac{\rho_2}{\rho_1}K_1)
\exp\{\xi_{22}t\}+
\frac{\W_{Tot}(0)-w^*}{\tau_1}+K_1.
\end{equation}


\subsubsection{Solution in region $\reg_4$}\label{Subsec:Solution in region 4}
Assuming that {$y(s)$ is in region $\reg_4$ for $s\in [0,t]$}, the system given by \eqref{eq:System ODE} can
be written as
\begin{equation}\label{eq:ODE_Vector}
y'(t)=\Xi y(t)^T.
\end{equation}
The eigenvalues of $\Xi$ are $\alpha_1=0$ and
$\alpha_2=\xi_{11}+\xi_{22}=\lambda_1+\mu_1(p_{11}-1)+\lambda_2+\mu_2(p_{22}-1)<0.$
Let $V_1$, $V_2$ be the corresponding eigenvectors; i.e.,
$V_1=
\begin{pmatrix}
1\\
-\frac{\xi_{11}}{\xi_{12}}
\end{pmatrix}$  and
$V_2=
\begin{pmatrix}
1\\
\frac{\alpha_2-\xi_{11}}{\xi_{12}}
\end{pmatrix}=
\begin{pmatrix}
1\\
\frac{\xi_{21}}{\xi_{11}}
\end{pmatrix}.$\\
Using the relations between the coefficients of matrix $\Xi$, the solution of \eqref{eq:ODE_Vector} is given by	
\begin{align}
y_1(t)=&c_2+c_2'\exp\{\alpha_2t\},\label{eq:solution y1 in r4}\\
y_2(t)=&c_2\frac{\rho_{2}}{\rho_{1}}
-c_4'\frac{\tau_{1}}{\tau_{2}}\exp\{\alpha_2t\},\label{eq:solution y2 in r4}\nonumber
\end{align}
with $c_2=\frac{\W_{Tot}(0)}{w^*}K_1$ and $c_2'=-\Big[y_2(0)-\frac{\rho_2}{\rho_1}y_1(0)\Big]
\Big[\frac{\tau_2\rho_1}
{\tau_1\rho_1+\tau_2\rho_2}\Big].$

Having found the solution of \eqref{eq:System ODE} in each region $\reg_3$ and $\reg_4$, we observe that the 2-dimensional equation can be reduced to a 1-dimensional equation since $\W_{Tot}(0)=\tau_1\x_1(t)+\tau_2\x_2(t),$ for $t\geq 0.$ Now, it is enough to show the convergence of this equation. To see this, define $x(\cdot)$ as
\begin{equation}\label{eq:reduced}
x(t):=y_2(t)=\frac{\W_{Tot}(0)}{\tau_2}-\frac{\tau_1\x_1(t)}{\tau_2},
\end{equation}
and its derivative as
\begin{equation}\label{eq:derivative of reduced}
x'(t)=-\frac{\tau_1\x'_1(t)}{\tau_2}.
\end{equation}

Using the observation that we can reduce the dimension by one and the solutions to system \eqref{eq:System ODE}, we show that the fluid model
solutions converge to an equilibrium state uniformly for all initial states within a compact set. First, we find the sign of the derivative of the above reduced equation. Then, as the limiting point depends on the sign
of the quantity $\W_{Tot}(0)-w^*,$ we have to distinguish between the following three cases: the total workload in the fluid model is i) greater than, ii) less than or iii) equal to the critical workload.


\subsection{Local analysis: establishing monotonicity}\label{Local analysis: establishing monotonicity}
\label{sec:The sign of the derivative of the reduced equation}
By \eqref{eq:reduced}, it is clear that we need to study only the behaviour of $y_1(t).$ In the sequel, we find the sign of \eqref{eq:derivative of reduced} in each region $\reg_i,$ for $i=1,\ldots,4$.

In region $\reg_1$, we know that $\x_1(t)\leq K_1$. By \eqref{eq:System ODE} and \eqref{eq:relation2}, we have that
\begin{equation*}\label{eq:bount in 1}
\begin{split}
x'(t)&=-\frac{\tau_1\x'_1(t)}{\tau_2}=-\frac{\tau_1}{\tau_2}(\xi_{11}\x_1(t)+\xi_{12}K_2)
=- \frac{\tau_1}{\tau_2}\xi_{12}(-\frac{\rho_2}{\rho_1}\x_1(t)+K_2)\\
&\leq -\frac{\tau_1}{\tau_2}\xi_{12}(-\frac{\rho_2}{\rho_1}K_1+K_2)=-\epsilon_1<0,
\end{split}
\end{equation*}
where $\epsilon_1=\frac{\tau_1}{\tau_2}\xi_{12}(-\frac{\rho_1}{\rho_2}K_1+K_2),$ which is strictly positive by \eqref{In:Bot.AS} and the fact that $\xi_{12}>0$.
Therefor we conclude that the derivative of $x(t)$ is strictly negative in region $\reg_1.$

In region $\reg_2,$ by \eqref{eq:System ODE} and \eqref{eq:relation2}, we have that
\begin{equation*}
\begin{split}
x'(t)
=-\frac{\tau_1}{\tau_2}(\xi_{11}K_1+\xi_{12}K_2)
=-\frac{\tau_1\xi_{12}}{\tau_2}(-\frac{\rho_2}{\rho_1}K_1+K_2)
=-\epsilon_1<0.
\end{split}
\end{equation*}
Thus, the derivative of $x(t)$ is strictly negative for all $\x_1(0)$ in regions $\reg_1$ and $\reg_2.$ In other words, the trajectory of $x(t)$ leaves regions $\reg_1$ and $\reg_2$ after a finite time. Now, we move to the regions where the invariant points lie, i.e., $\reg_3$ and $\reg_4$.

In region $\reg_3$, by \eqref{eq:critical} and \eqref{eq:explicit solution y'1 in region 3}, we obtain
\begin{equation*}
\begin{split}
x'(t)=\frac{\tau_1}{\tau_2}
\xi_{22}(\frac{\W_{Tot}(0)}{\tau_1}-y_1(0)-\frac{\tau_2}{\tau_1}\frac{\rho_2}{\rho_1}K_1)
\exp\{\xi_{22}t\}
= \frac{\tau_1}{\tau_2}
\xi_{22}(\frac{\W_{Tot}(0)-w*}{\tau_1}+K_1-y_1(0))
\exp\{\xi_{22}t\}.
\end{split}
\end{equation*}
We saw in \eqref{eq:relation1} that $\xi_{22}<0$. We therefore have that in $\reg_3$,
\begin{equation}\label{derivative in 3}
x'(t)= \twopartdef { >0} {y_1(0)>\frac{\W_{Tot}(0)-w*}{\tau_1}+K_1,} {<0} {y_1(0)<\frac{\W_{Tot}(0)-w*}{\tau_1}+K_1,}
\end{equation}
and $x'(t)=0$ if $y_1(0)=\frac{\W_{Tot}(0)-w*}{\tau_1}+K_1.$

In region $\reg_4,$ by \eqref{eq:critical} and \eqref{eq:solution y1 in r4}, we have that
\begin{equation*}
\begin{split}
x'(t)&=-\frac{\tau_1\x'_1(t)}{\tau_2}
=-\frac{\tau_1}{\tau_2}c_4'\alpha_2\exp\{\alpha_2t\},
\end{split}
\end{equation*}
where
\begin{equation*}
\begin{split}
c_4'=
-\Big[\frac{\W_{Tot}(0)}{\tau_2}-\frac{\tau_1}{\tau_2}y_1(0)
-\frac{\rho_2}{\rho_1}y_1(0)\Big]
\Big[\frac{\tau_2\rho_1}
{\tau_1\rho_1+\tau_2\rho_2}\Big]
&=-\Big[\frac{\W_{Tot}(0)}{\tau_2}-\frac{\tau_1\rho_1+\tau_2\rho_2}{\tau_2\rho_1}y_1(0)\Big]
\Big[\frac{\tau_2\rho_1}
{\tau_1\rho_1+\tau_2\rho_2}\Big]\\
&=-\Big[\W_{Tot}(0)-
\frac{w^*}{K_1}y_1(0)\Big]
\Big[\frac{\rho_1}
{\tau_1\rho_1+\tau_2\rho_2}\Big].
\end{split}
\end{equation*}
Recall that from Section~\ref{Subsec:Solution in region 4}, we have that $\alpha_2<0$ and thus in $\reg_4$,
\begin{equation}\label{derivative in 4}
x'(t)= \twopartdef { >0} {y_1(0)>\frac{\W_{Tot}(0)}{w*}K_1,} {<0} {y_1(0)<\frac{\W_{Tot}(0)}{w*}K_1,}
\end{equation}
and $x'(t)=0$ if $y_1(0)=\frac{\W_{Tot}(0)}{w*}K_1.$

Combining \eqref{derivative in 3} and \eqref{derivative in 4}, and keeping in mind that $\Delta_1(w)=\frac{\min\{w,w^\ast\}}{w^\ast}K_1+\frac{\max\{w-w^\ast,0\}}{\tau_1}$ we have that in $\reg_3 \cup \reg_4$,
\begin{equation*}\label{derivative in 4}
x'(t)= \twopartdef { >0} {y_1(0)>\Delta_1(\W_{Tot}(0)),} {<0} {y_1(0)<\Delta_1(\W_{Tot}(0)),}
\end{equation*}
and $x'(t)=0$ if $y_1(0)=\Delta_1(\W_{Tot}(0)).$



\subsection{Global analysis: convergence to Invariant Manifold}\label{Convergence to Invariant Manifold}

We are now ready to connect all pieces. From the previous section, we know that $x(t)=y_2(t)$ must be smaller than $K_2$ after a finite time, {thus exiting regions $\Pi_1$ and $\Pi_2$}. Therefore, we can focus on the remaining two regions. In order to do so, we need to consider whether $y_1(t)$ will eventually be larger than, smaller than, or equal to $K_1$. This leads to three cases, treated separately in the remainder of this section.

\phantomsection
\addcontentsline{toc}{subsubsection}{Case 1}
\paragraph{Case 1: $\W_{Tot}(0)-w^*> 0$.} In this case the invariant point (limiting point) lies in region $\reg_3.$ If $\x_1(0)\in \reg_3,$ then we know that $\x_1(0)\geq K_1.$ By \eqref{eq:explicit solution y'1 in region 3}, we have that
\begin{equation*}
\begin{split}
y_1(t)=&-\frac{\tau_2}{\tau_1}
(\frac{\W_{Tot}(0)}{\tau_2}-\frac{\tau_1 \x_1(0)}{\tau_2}-\frac{\rho_2}{\rho_1}K_1)
\exp\{\xi_{22}t\}+
\frac{\W_{Tot}(0)-w^*}{\tau_1}+K_1\\
\geq & (- \frac{\W_{Tot}(0)}{\tau_1}+K_1+\frac{\tau_2\rho_2}{\tau_1\rho_1}K_1)\exp\{\xi_{22}t\}+\frac{\W_{Tot}(0)-w^*}{\tau_1}+K_1\\
\geq & (- \frac{\W_{Tot}(0)}{\tau_1}+K_1+\frac{w*}{\tau_1}-K_1)\exp\{\xi_{22}t\}+ \frac{\W_{Tot}(0)-w^*}{\tau_1}+K_1\\
\geq & \frac{\W_{Tot}(0)-w*}{\tau_1}(1-\exp\{\xi_{22}t\})+K_1 \geq K_1.
\end{split}
\end{equation*}

If $\x_1(0)\in \reg_4,$ then $\x_1(0)\leq K_1.$ Also, by the assumption that $\W_{Tot}(0)-w^*> 0$ and the definition of the lifting map \eqref{d2}, we have that $\x_1(0)\leq K_1<\Delta_1\W_{Tot}(0).$ This implies that $x(0)\geq \frac{\W_{Tot}(0)}{\tau_2}-\frac{\tau_1K_1}{\tau_2}$ and $x(t)$ is strictly decreasing. In the sequel, we show that there exists a time $t^*$ such that $\x_1(t^*)=K_1.$ This means that the function $x(t)$ lies in region $\reg_3$ after that time.
It is enough to prove that the equation $\x_1(t)=K_1,$ has a positive solution.
Note that by \eqref{eq:solution y1 in r4}, we have that
\begin{equation*}
\begin{split}
y_1(t)&=\frac{\W_{Tot}(0)}{w^*}K_1
-\Big[\W_{Tot}(0)-
\frac{w^*}{K_1}y_1(0)\Big]
\Big[\frac{\rho_1}
{\tau_1\rho_1+\tau_2\rho_2}\Big]
\exp\{\alpha_2 t\}.
\end{split}
\end{equation*}
Now, setting $\x_1(t)= K_1,$ the previous equation becomes
\begin{equation}\label{eq:t*}
\begin{split}
\frac{\W_{Tot}(0)}{w^*}K_1-K_1 &=
\Big[\W_{Tot}(0)-
\frac{w^*}{K_1}y_1(0)\Big]
\Big[\frac{\rho_1}
{\tau_1\rho_1+\tau_2\rho_2}\Big]
\exp\{\alpha_2 t\}.
\end{split}
\end{equation}
We argue that \eqref{eq:t*} is satisfied by a $t\geq0$ as follows. Observe that by \eqref{eq:critical}
the quantity
$\frac{\rho_1}{\tau_1\rho_1+\tau_2\rho_2} $
is equal to $\frac{K_1}{w^*},$ and by the assumption that $\x_1(0)\in \reg_4$, we have that $\x_1(0)\leq K_1$. Now, we can obtain the following inequality
\begin{equation*}
\Big[\W_{Tot}(0)-
\frac{w^*}{K_1}y_1(0)\Big]
\Big[\frac{\rho_1}
{\tau_1\rho_1+\tau_2\rho_2}\Big]\geq
\Big[ \frac{\W_{Tot}(0)}{w^*}K_1-K_1\Big],
\end{equation*}
which proves the statement. Note that if the total workload is equal to the critical workload, then \eqref{eq:t*} does not have a positive solution since \eqref{eq:t*} would imply that $\exp\{\alpha_2 t\}=0$. It would only be satisfied by $t=0$. This means, that if $\x_1(0)\in \reg_4$ and $\W_{Tot}(0)-w^*=0,$ then the function $x(t)$ remains in region $\reg_4$ for ever.
Since in this first case $\W_{Tot}(0)-w^*> 0$, we have that $\frac{\W_{Tot}(0)}{w^*}K_1-K_1>0$. Thus, by combining the fact that $\x_1(0)\leq K_1$ and the previous display, we have shown that
$\W_{Tot}(0)-
\frac{w^*}{K_1}y_1(0)> \W_{Tot}(0)-w^*>0.$ Combining these arguments leads to the conclusion that the equation $\x_1(t^*)=K_1$ has a (unique) positive solution, say $t^*$. Therefore, $y(t)\in \reg_3$ for $t>t^*$. In other words, it is enough to prove that for $\W_{Tot}(0)-w^*>0$, the function $x(t)$ converges to a point in region $\reg_3$.

We now show that it converges to the invariant manifold in region $\reg_3$. By \eqref{d2}, we have that $\Delta_2\W_{Tot}(0)=\frac{\rho_2}{\rho_1}K_1.$ By \eqref{eq:explicit solution y'1 in region 3} and \eqref{eq:reduced},
\begin{equation*}
|x(t)-\Delta_2\W_{Tot}(0)|\leq \big|(\frac{\W_{Tot}(0)}{\tau_2}-\frac{\tau_1 \x_1(0)}{\tau_2}-\frac{\rho_2}{\rho_1}K_1)\big|
\exp\{\xi_{22}t\}
\end{equation*}
and recall that $\xi_{22}<0.$ Also, for any closed and bounded interval of the form $[0,M]$ for $M>0$, we have that the quantity $|(\frac{\W_{Tot}(0)}{\tau_2}-\frac{\tau_1 \x_1(0)}{\tau_2}-\frac{\rho_2}{\rho_1}K_1)|$ is uniformly bounded by
$\sup_{\x_1(0)\in [0,M]} |(\frac{\W_{Tot}(0)}{\tau_2}-\frac{\tau_1 \x_1(0)}{\tau_2}-\frac{\rho_2}{\rho_1}K_1)|.$ That is, the convergence is uniform for any initial state in a compact set.

\phantomsection
\addcontentsline{toc}{subsubsection}{Case 2}
\paragraph{Case 2: $\W_{Tot}(0)-w^*< 0.$}
Adapting the previous case, we first show that if $\W_{Tot}(0)-w^*<0$ and $\x_1(0)\in \reg_4,$ the the function $x(t)$ remains for ever in region $\reg_4.$ To see this,
by \eqref{eq:solution y1 in r4} we have that
\begin{equation*}
\begin{split}
y_1(t)&=\frac{\W_{Tot}(0)}{w^*}K_1
-\Big[\W_{Tot}(0)-
\frac{w^*}{K_1}y_1(0)\Big]
\frac{K_1}
{w^*}
\exp\{\alpha_2 t\}.
\end{split}
\end{equation*}
Observing that $\x_1(0)\leq K_1$ in $\reg_4$, we derive the following inequality
\begin{equation*}
\begin{split}
y_1(t)\leq \frac{\W_{Tot}(0)}{w^*}K_1
+\Big[(1-\frac{\W_{Tot}(0)}{w^*})K_1\Big]
\exp\{\alpha_2 t\}.
\end{split}
\end{equation*}
Note that the term $(1-\frac{\W_{Tot}(0)}{w^*})$ is positive by the assumption $\W_{Tot}(0)-w^*< 0$ and that $\exp\{\alpha_2 t\}\leq 1$. Combining these  three facts, we have that $y_1(t)\leq K_1$.

Now, we show that if the process starts in region $\reg_3$ and $\W_{Tot}(0)-w^*< 0,$ then there exists a finite time $t^{**}$ such that $x(t)\in \reg_4$ after that time. Again, here we prove that the equation $\x_1(t^*)=K_1$ has a positive solution. Then, the result follows by observing that $\x_1(0)\geq K_1>\Delta_1\W_{Tot}(0),$ and in that case $x(t)$ is an increasing function. By \eqref{eq:explicit solution y'1 in region 3}, we have that
\begin{equation*}
\begin{split}
y_1(t)&=-\frac{\tau_2}{\tau_1}
(\frac{\W_{Tot}(0)}{\tau_2}-\frac{\tau_1\x_1(0)}{\tau_2}-\frac{\rho_2}{\rho_1}K_1)
\exp\{\xi_{22}t\}+
\frac{\W_{Tot}(0)-w^*}{\tau_1}+K_1\\
&= -\big(\frac{\W_{Tot}(0)-w^*}{\tau_1}+K_1-\x_1(0)\big)\exp\{\xi_{22}t\}+\frac{\W_{Tot}(0)-w^*}{\tau_1}+K_1.
\end{split}
\end{equation*}
Setting $\x_1(t)=K_1,$ we obtain
\begin{equation}\label{eq:t**}
\begin{split}
\frac{\W_{Tot}(0)-w^*}{\tau_1}&=\big(\frac{\W_{Tot}(0)-w^*}{\tau_1}+K_1-\x_1(0)\big)\exp\{\xi_{22}t\}.
\end{split}
\end{equation}
To show that the previous equation has a positive solution, it is enough to show that
\begin{equation*}
\begin{split}
\frac{\tau_1}{\W_{Tot}(0)-w^*}\big(\frac{\W_{Tot}(0)-w^*}{\tau_1}+K_1-\x_1(0)\big)
\geq 1.
\end{split}
\end{equation*}
Recall that $\x_1(0)\geq K_1$ and that$\W_{Tot}(0)-w^*<0$. We can now derive the previous inequality by observing that
\begin{equation*}
\begin{split}
\frac{\tau_1}{\W_{Tot}(0)-w^*}\big(\frac{\W_{Tot}(0)-w^*}{\tau_1}+K_1-K_1\big)\geq 1.
\end{split}
\end{equation*}
Analogously with the previous case, we note that if the total workload is equal to the critical workload, \eqref{eq:t**} does not have a positive solution. This means, that if $\x_1(0)\in \reg_3$ and $\W_{Tot}(0)-w^*=0,$ then the function $x(t)$ remains in region $\reg_3$ for ever.

For the convergence to the invariant manifold, we have that $\Delta_2\W_{Tot}(0)=\frac{\rho_2}{\rho_1}\frac{\W_{Tot}(0)}{w^*}K_1$ by \eqref{d2}. Moreover, \eqref{eq:solution y1 in r4} and \eqref{eq:reduced} lead to
\begin{equation*}
|x(t)-\Delta_2\W_{Tot}(0)|\leq
\big|\frac{\tau_1}{\tau_2}(\frac{\W_{Tot}(0)w^*}{w^*K_1}-
y_1(0))\big|
\exp\{\alpha_2t\}
\end{equation*}
and recall that $\alpha_2 <0$. The uniform convergence in a compact set for any initial state follows for the same reason as in previous case.

\phantomsection
\addcontentsline{toc}{subsubsection}{Case 3}
\paragraph{Case 3: $\W_{Tot}(0)-w^*= 0.$}
In this case, the convergence follows from the comments we made in the previous two cases. If $\W_{Tot}(0)-w^*= 0,$ then the function $x(\cdot)$ always stays in the region where $y_1(0)$ lies (see comments after \eqref{eq:t*} and \eqref{eq:t**}). As we see, the function converges in regions $\reg_3$ and $\reg_4.$\\

This concludes the proof of Theorem~\ref{Thm: Convergence of the fluid model}, which will be applied to prove a diffusion theorem for the queue length process in the next section.


\section{Diffusion approximations}\label{Sec: Diffusion approximations}
The main objective in this section is to show a state-space collapse property (SSC) for the diffusion queue length process. This yields a diffusion limit theorem for the diffusion-scaled process. To do it, we follow the strategy set up in \citep{bramson1998state}. Let us consider a family of single-server systems indexed by $n\in\mathbb{N},$ where $n$ tends to infinity, with the same basic structure as that of the network described in Section~\ref{sec: Model}. To indicate the position in the sequence of networks, a superscript $n$ will be appended to the network parameters and processes. Diffusion (or central limit theorem) scaling is indicated by placing a hat over a process.
Thus, the well-known diffusion scaling is given by $\wh{\qnp}^n(\cdot)=\frac{1}{n}\qnp^n(n^2\cdot).$
Let $\rho^n=\sum_{i=1}^{\I}\rho_i^n=(\gamma^n)^T\beta$. We set $\gamma^n=(I-P^T)^{-1}\lambda^n$, $\lnn=\lambda(1-\frac{\theta}{n})$, $\mu^n\equiv\mu$, $P^n\equiv P$, and $K_i^n=nK_i,$ where $\theta$ is a positive real number. Thus, we have that
$
\rho^n=1-\frac{\theta}{n}.
$
It is clear that under the critical loading assumption, $\lnn\rightarrow \lambda$ and $n(1-\rho^n)\rightarrow \theta,$ as $n\rightarrow \infty$. These are our heavy traffic assumptions. Furthermore, we assume that $\ba{Q}^n(0)=\frac{1}{n}Q^n(0)\rightarrow \ba{Q}(0),$ where $\ba{Q}(0)$ is a positive constant.
The service allocation function for the $n^{\text{th}}$ model is given by
\begin{equation*}\label{Rfunction}
R^n_i(q)= \twopartdef {  \frac{\min_i\{ {q_i},{nK_i}\}}{\sum_{j=1}^{I}\min\{ {q_j},{nK_j}\}}} {q_i\neq0,} {0} {q_i=0,}
\end{equation*}
where we observe that $R^1_i(\cdot)=R_i(\cdot)$. Recall that $R(\cdot)$ is a Lipschitz-continuous function on $\mathbb R^2_+\backslash\{0\}$ and observe that the following scaling property holds: $R^n(n\cdot)=R(\cdot)$.
In the sequel, we state the technical assumptions that allow us to apply the functional central limit theorem and Bramson's weak law estimates. We assume that for $i=1,\I$,
\begin{equation*}\label{eq:assumptions}
\frac{u_i^n(1)}{n}\rightarrow 0 \quad \text{and} \quad\frac{v^{(1),n}_{ii}(1)}{n}\rightarrow 0,
\end{equation*}
in probability as $n \rightarrow \infty$. In addition, we assume that there exists a function $\eta(\cdot),$ with zero limit at infinity such that for $j>1,$ $i,l\in \{1,2\}$ and $k\in \mathbb{N},$
\begin{equation*}
\E(u_i^n(j)^2 \ind {u_i^n(j)>a})\le\eta(a)  \quad \text{and} \quad
\E(v^{(k),n}_{il}(j)^2 \ind {v^{(k),n}_{il}(j)>a})\le\eta(a).
\end{equation*}
More details about these assumption can be found in \citep{bramson1998state} and \citep{ye2012stochastic}. In the rest of this section, we assume that the previous assumptions are satisfied without refer to them again. The main result in this section is the following
\begin{thm}\label{Thm:main result}
Assume that the diffusion-scaled initial state converges in distribution as $n\to \infty$, i.e.\
$\wh{Q}^n(0) \overset{d} \rightarrow \Delta\wh{W}_{Tot}(0),$
where ``$\overset{d}\rightarrow$" denotes convergence in distribution. Then, the diffusion-scaled stochastic process converges in distribution as $n\to \infty$, i.e.\
\begin{equation*}
\wh{Q}^n(\cdot) \overset{d}\rightarrow \Delta\wh{W}_{Tot}(\cdot),
\end{equation*}
where $\wh{W}_{Tot}(t)$ is a 1-dimensional Brownian motion with drift $-\theta$ and variance
$\sigma^2=\sum_{i=1}^{\I}\lambda_i\sigma_{s_i}^2 + \tau_i^2\lambda_i c_{u_i}^2,$
where $\sigma_{s_i}^2$ denotes the variance of $s_i(j)$ and $c_{u_i}$ denotes the coefficient of variation of $u_i(j).$
\end{thm}
The proof of this theorem is given in the end of this section. This theorem can be applied to develop heavy traffic approximations for the joint queue length process, as they are both a piece-wise linear function of a one-dimensional RBM, of which the time-dependent distribution
can be expressed in closed form (in terms of the Gaussian cdf and pdf); cf.\ \cite{chen2001fundamentals}.  To see this, note that the functions $\Delta_1$ and $\Delta_2$ are invertible with inverses
\begin{equation*}
\begin{split}
\Delta_1^{-1}(w)&= \twopartdef {\frac{w w^*}{K_1}} {w<w^*,} {\tau_1(w-K_1)+w^*} {w \geq w^*,}\\
\Delta_2^{-1}(w)&= \twopartdef {\frac{\rho_1 ww^*}{\rho_2K_1}} {w<w^*,} {\infty}  {w \geq w^*.}
\end{split}
\end{equation*}
We know that $\wh{W}_{Tot}(t)$ is a RBM$(-\theta,\sigma^2).$ Let $\wh{Q}(\cdot)$ be the diffusion limit. We have that for $x,y\geq 0,$
\begin{equation*}
\begin{split}
\Prob{(\wh{Q}_1(t)>x, \wh{Q}_2(t)>y)}=
\Prob{(\Delta_1\wh{W}_{Tot}(t)>x, \Delta_2\wh{W}_{Tot}(t)>y)}
=&\Prob{(\wh{W}_{Tot}(t)>\Delta_1^{-1}(x),
\wh{W}_{Tot}>
\Delta_2^{-1}(y))}\\
=&\Prob{(\wh{W}_{Tot}(t)>z)},
\end{split}
\end{equation*}
where $z=\max \{\Delta_1^{-1}(x),\Delta_2^{-1}(y) \}.$ The last expression can be written in terms of the Gaussian cdf and pdf; cf.\ \cite{chen2001fundamentals}.
Also, using a similar coupling argument as in \cite{zhang2008steady}, it can be shown that one can interchange the steady-state and heavy traffic limits in this case.
For space considerations we will leave this as detail to the reader.

The rest of this section is devoted to a proof of Theorem \ref{Thm:main result}. It is organized as follows.
\begin{enumerate}
\item We first prove a heavy traffic limit theorem for the total workload process.
\item After that, we define a family of shifted fluid-scaled processes in Section~\ref{Shifted fluid-scaled processes} and we show that they are stochastically bounded in Section~\ref{Bounding the shifted processes}.
\item In Section~\ref{Uniform fluid approximation}, we establish some technical auxiliary estimates and tightness of these families. Moreover, we establish that limit points of these fluid scaled processes, which are called fluid limits, are in fact fluid model solutions as defined in Section~\ref{sec: Fluid analysis}. The development in this section is very similar to those in Bramson \cite{bramson1998state} and is therefore  kept concise.
\item In Section~\ref{The scaled shifted total workload process}, we establish a similar tightness property for a family of shifted fluid-scaled workload processes.
\item The proof is then completed by showing a state-space collapse result in Section~\ref{State-space collapse}.
\end{enumerate}

\subsection{Convergence of the total workload}\label{Convergence of total workload}

\begin{lem}\label{l1}
Under the critical loading assumption, the diffusion-scaled total workload, $\wh{W}^n_{Tot}(t)=\frac{1}{n}\wh{W}^n_{Tot}(n^2t)$, converges in distribution to a RBM$(-\theta,\sigma^2)$.
\end{lem}
\begin{proof}
By \eqref{eq:total S requirement} for $i,l=1,\I$ and $i\neq l,$ we have that the total service requirement of the $j^\text{th}$ external customer (including customers who already are in queue $i$ at time zero) who enter at queue $i$ is given by
\begin{equation*}
s_i(j)=v_{ii}^{(1)}(j)+
\sum_{k=1}^{\infty}
\f^{(k)}_{ii}(j)v_{ii}^{(k+1)}(j)+
\f^{(k)}_{il}(j)v_{il}^{(k+1)}(j).
\end{equation*}
We define the following process
\begin{equation}\label{eq:second definition of Total W}
W_{G}(t):=\sum_{i=1}^{\I}\sum_{j=1}^{Q_i(0)+E_i(t)}
s_i(j)-
\int_{0}^{t}\ind{W_{G}(s)> 0}ds.
\end{equation}
We recall that $E_i(\cdot)$ denotes the external arrivals at queue $i$ and by construction of the model, $\{s_i(j)\}_{j=1}^{\infty}$ is a sequence of positive i.i.d.\ random variables for $i=1,\I$. The process given in \eqref{eq:second definition of Total W} represents the workload of a single queue with input given by two independent renewal process that have independent service requirements from each other. The busy time of this system is $\int_{0}^{t}\ind{W_{G}(s)> 0}ds=T_1(t)+T_2(t)$,
and it represents the busy time of server in layer 2. Note that the busy time is zero if and only if both queues are empty.
The diffusion-scaled process (after subtracting and adding the means of the random quantities in \eqref{eq:second definition of Total W}), is given by
\begin{equation*}
\begin{split}
\wh{W}^n_{G}(t)&=\sum_{i=1}^{\I}\frac{1}{n}\Big(\sum_{j=1}^
{n^2(\ba{Q}^n_i(0)+\ba{E}^n_i(t))}
s^n_i(j)-\tau_in^2(\ba{Q}^n_i(0)+\ba{E}^n_i(t))\Big)\\
&{\hskip 5pt}+\tau_i(\wh{E}^n_i(t)-\lambda^n_int)
+\tau_i\ba{Q}^n_i(0)+\tau_i\lambda^n_int-
nt+\wh{Y}^n(t),
\end{split}
\end{equation*}
where $\wh{Y}^n(t)=\int_{0}^{t}\ind{\wh{W}^n_{G}(s)= 0}ds$.
By the time change theorem in \citep{whitt1980some} and the functional central limit theorem (also see \cite[Theorem~6.8]{chen2001fundamentals}), we have that
$\wh{W}^n_{G}(\cdot)\overset{d}\rightarrow \wh{W}_{G}(\cdot),$ as $n\rightarrow \infty$. Furthermore, the limit can be described as
\begin{equation*}
\begin{split}
\wh{W}_{G}(t)
=\sum_{i=1}^{\I}\tau_i \ba{Q}_i(0)-\theta t
+ \sqrt{\lambda_1\sigma_{s_1}^2} \mathcal{W}_1(t)
+\sqrt{\lambda_2\sigma_{s_2}^2} \mathcal{W}_2(t)
+\tau_1\sqrt{\lambda_1c_{u_1}^2}  \mathcal{W}_3(t)
+\tau_2\sqrt{\lambda_2c_{u_2}^2} \mathcal{W}_4(t)+\wh{Y}(t),
\end{split}
\end{equation*}
where $\mathcal{W}_i(t),$ $1 \leq i \leq 4$ are independent 1-dimensional standard Brownian motions
and $\wh{Y}(t)$ can be increased only if $\wh{W}_{G}(t)=0.$
Thus, the process $\wh{W}_{G}(t)$ satisfies a 1-dimensional Skorokhod problem. That is, $\wh{W}_{G}(t)$ is a reflected Brownian motion starting at point $\sum_{i=1}^{\I}\tau_i \ba{Q}_i(0)$ with drift $-\theta$ and variance $\sigma^2$ which is given by
$
\sigma^2=\sum_{i=1}^{\I}\lambda_i\sigma_{s_i}^2 + \tau_i^2\lambda_i c_{u_i}^2,
$
where $\sigma_{s_i}^2$ denotes the variance of $s_i(j)$ and $c_{u_i}$ denotes the coefficient of variation of $u_i(j)$. The second moment of the random variables $s_i(j)$ is given by \eqref{eq:secont moment of Total time}. In case of Poisson external arrivals, this result is reduced to the well-known heavy-traffic limit (see e.g.\ \cite[Theorem~2.3]{gromoll2004diffusion}).

Now, we shall prove that
\begin{equation}\label{eq:equality of workloads}
W_{Tot}(t)=W_{G}(t),{\hskip 10pt}t\geq0.
\end{equation}
We do it by showing that we can change the label of how we count the service requirements of the customers in the system. Counting the total service requirements of the external arrivals until time $t$ is the same as counting the immediate and remaining service requirements of the total arrivals in the system until time $t$.
Recall that $s_i(j)=v_{ii}^{(1)}(j)+s'_i(j)$ for $i=1,\I$, where $s'_i(j)$ are the future service requirements.
If $t=0,$ then we have nothing to prove as $E_i(0)=A_i(0)=0$. If $t>0$ and $\Phi_{li}\Big(S_l\big(T_l(t)\big)\Big)=0$ for $l,i\in \{1,2\},$ then \eqref{eq:equality of workloads} holds as all the departures until time $t$ leave the system and so $s'_i(j)=0$ for $j=1,\ldots,S_i\big(T_i(t)\big)$ and $E_i(t)=A_i(t)$.

For the general case, we first assume that all customers are routed only one time until the time $t.$ The right-hand side in \eqref{eq:second definition of Total W} can be written as
\begin{equation}\label{alternative writing}
\begin{split}
\sum_{i=1}^{\I}\sum_{j=1}^{Q_i(0)+E_i(t)}
v_{ii}^{(1)}(j)+
\sum_{j=S_i\big(\mu_iT_i(t)\big)+1}^{Q_i(0)+E_i(t)}
s'_i(j)+
\sum_{j=1}
^{S_i\big(T_i(t)\big)}
s'_i(j).
\end{split}
\end{equation}
In order to separate the customers who depart form node $i$ (if they are routed or leave the system), we define the following sets for $i,l\in \{1,\I\}$ and $i\neq l,$
$\mathcal{A}^i=\{j:1\leq j \leq S_i\big(T_i(t)\big)\},$
which includes all the customers who depart from node $i.$ Customers can be routed at the same node and we denote this set by
$\mathcal{A}^i_1=\{j\in \mathcal{A}^i:\varphi_{ii}(j)=1\}$
or to the other node and we denote this set by
$\mathcal{A}^i_2=\{j\in \mathcal{A}^i:\varphi_{il}(j)=1\}.$
Last, we define the following set
$\mathcal{A}^i_3=\mathcal{A}^i \setminus(\mathcal{A}^i_1\cup \mathcal{A}^i_2),$
which represents all  customers who leave the system.
If $j\in \mathcal{A}^i_1,$ then there exist natural numbers $k_j^i$ such that
$S_i\big(T_i(t)\big)\leq k_j^i \leq A_i(t)-1$ and
\begin{equation*}
s'_i(j)=v_{ii}^{(1)}(k_j^i+1)+s'_i(k_j^i+1).
\end{equation*}
Similarly, if $j\in \mathcal{A}^i_2$, then for $i\neq l$ there exist $h_j^i,$ such that $S_l\big(T_l(t)\big)\leq h_j^i \leq A_l(t)-1$ and
\begin{equation*}
s'_i(j)=v_{ll}^{(1)}(h_j^i+1)+s'_l(h_j^i+1).
\end{equation*}
Last, $j\in \mathcal{A}^i_3$ means that the $j^\text{th}$ customer leaves the system  after his first service and so $s'_i(j)=0$.
The quantities $k_j^i$ and $h_j^i$ denote the number of customers in node $i=1,\I$ (including that one in service) who the
$j^{\textrm{th}}$ customer meets after his departure from node $l=1,\I$.
Therefore, \eqref{alternative writing} can be written as
\begin{equation*}\label{alternative 2}
\begin{split}
\sum_{i=1}^{\I}\sum_{j=1}^{Q_i(0)+A_i(t)}
v_{ii}^{(1)}(j)+
\sum_{j=S_i\big(T_i(t)\big)+1}^{Q_i(0)+A_i(t)}
s'_i(j).
\end{split}
\end{equation*}
Now, let $k_{0}^i(j)$ be the number of routes at node $i,$ for $1 \leq j \leq Q_i(0)+E_i(t)$ until time $t.$ Also, we define the number $k_0^i=\max_{j}k_{0}^i(j),$ and the number of maximum routes for any external arrival in the system until time $t,$
$k_0=\max_{i=1,2}\{k_0^i\}.$
Observe that for any $t\geq 0,$ $k_0\in \mathbb{N}$ and it is finite because we assume Markov routing.
Take the following partition
$[0,t]= \bigcup_{m=0}^{k_0} [t_m,t_{m+1}],$
where $t_0=0$ and $t_{{k_0+1}}=t$. We take the previous partition in such way, so that in each interval $[t_m,t_{m+1}],$ any customer in the system can be routed only one time and so that in the interval $[t_{k_0},t]$ there is no routing.
Now, we can write the right-hand side of \eqref{eq:second definition of Total W} as
\begin{equation*}
\begin{split}
\sum_{i=1}^{\I}\sum_{j=1}^{Q_i(0)+E_i(t)}
v_{ii}^{(1)}(j)+
\sum_{j=S_i\big(T_i(t_1)\big)+1}^{Q_i(0)+E_i(t)}
s'_i(j)+
\sum_{j=1}
^{S_i\big(T_i(t_{1})\big)}
s'_i(j).
\end{split}
\end{equation*}
Applying the previous idea where customers are routed only one time per interval, we have that the above quantity can be written as
\begin{equation*}
\begin{split}
\sum_{i=1}^{\I}\sum_{j=1}^{Q_i(0)+E_i(t)+\sum_{l=1}^{\I}\Phi_{li}\big(S_l(T_l(t_1))\big)}
v_{ii}^{(1)}(j)+
\sum_{j=S_i\big(T_i(t_1)\big)+1}^{Q_i(0)+E_i(t)+\sum_{l=1}^{\I}\Phi_{li}\big(S_l(T_l(t_1))\big)}
s'_i(j)
\end{split}
\end{equation*}
Split again the last term of the previous quantity until time $t_2,$ and apply the previous idea when customers are routed only one time to obtain
\begin{equation*}
\begin{split}
\sum_{i=1}^{\I}\sum_{j=1}^{Q_i(0)+E_i(t)+\sum_{l=1}^{\I}\Phi_{li}\big(S_l(T_l(t_2))\big)}
v_{ii}^{(1)}(j)+
\sum_{j=S_i\big(T_i(t_2)\big)+1}^{Q_i(0)+E_i(t)+\sum_{l=1}^{\I}\Phi_{li}\big(S_l(T_l(t_2))\big)}
s'_i(j).
\end{split}
\end{equation*}
Adapting the previous steps until time $t_{k_0},$ and recalling that the total arrival process is given by \eqref{TotalArPr}, we derive \eqref{eq:equality of workloads}.
\end{proof}

In the sequel, this result plays a key role. The next step is to define the so-called \textit{shifted fluid-scaled processes} and to show that they are stochastically bounded.

\subsection{Shifted fluid-scaled processes}\label{Shifted fluid-scaled processes}

We introduce the shifted fluid scaling, which is an extension of the classical fluid scaling. Let $T>0$ and $m\leq nT$. We define
\begin{equation*}
\bar{A}_i^{n,m}(t)=\frac{1}{n}\Big(A_i^n(nm+nt)-A_i^n(nm)\Big),
\end{equation*}
and the analogous scaling for the processes $T(\cdot)$, $Y_{L_2}(\cdot)$, and $E(\cdot)$.
For the departure process, the cumulative service time process, and the routing process we have that
\begin{equation*}
\bar{S}_i^{n,m}(t)=\frac{1}{n}\Big(S_i^n\big(nt +T_i^n(nm)\big)-S_i^n\big(T_i^n (nm)\big)\Big),
\end{equation*}
\begin{equation*}
\bar{V}^{n,m}_i(k)=\frac{1}{n}\Big(V_i^n\big(nk+S_i^n(T_i^n(nm))\big)-T_i^n(nm)\Big),
\end{equation*}
\begin{equation*}
\bar{\Phi}_{li}^{n,m}(k)=\frac{1}{n}
\Phi^n_{li}\Big(nk+S_l^n\big(T_l^n(nm)\big)\Big)-
\frac{1}{n}\Phi^n_{li}\Big(S_l^n\big(T_l^n(nm)\big)\Big).
\end{equation*}
Last, the queue length process is scaled as follows
$\bar{Q}_i^{n,m}(t)=\frac{1}{n}Q_i^n(nm+nt)$
and analogously for the scaling of the immediate and total workload.
The system dynamics \eqref{a}--\eqref{e} under the shifted fluid scaling become
\begin{align}
&\ba{Q}^{n,m}_i(t)=\ba{Q}^{n,m}_i(0)+
\bar{E}^{n,m}_i( t)
+\sum_{l=1}^{\I}\ba{\Phi}_{li}^{n,m}
\Big(\ba{S}^{n,m}_l
\Big(\ba{T}^{n,m}_l(t)\Big)
\Big)
-\ba{S}^{n,m}_i
\Big(\ba{T}^{n,m}_i(t)\Big),\label{RWa}\\
&\ba{T}^{n,m}_i(t)=
\int_{0}^{t} R^n_i(n\ba{Q}^{n,m}(s))ds
=\int_{0}^{t} R_i(\ba{Q}^{n,m}(s))ds,\label{RWc}\\
&\ba{W}^{n,m}_i(t)=\ba{V}^{n,m}_i
\bigg(
\ba{Q}^{n,m}_i(0)
+\ba{E}^{n,m}_i(t)+
\sum_{l=1}^{\I}
\ba{\Phi}_{li}^{n,m}
\Big(\ba{S}^{n,m}_l
\Big(\ba{T}^{n,m}_l(t)\Big)
\Big)
\bigg)
-\ba{T}^{n,m}_i(t),\label{RWb}\\
&\sum_{i=1}^{\I}\bar{T}_i^{n,m}(t)+\bar{Y}^{n,m}_{L_2}(t)=t, \label{RWd}\\
&\bar{Y}^{n,m}_{L_2}(t)\textrm{ increases }\Rightarrow \W_1^{n,m}(t)+\W_2^{n,m}(t)=0, \qquad t\geq0,\ i=1,\I. \label{RWe} \nonumber
\end{align}
In the sequel, we shall be referring to \textit{shifted fluid scaling}, and \textit{shifted fluid process} as \textit{shifted scaling}, and \textit{shifted process} for simplicity.
The main step of SSC is to show the shifted process can be approximated by
a solution of the fluid model. This is done in Section~\ref{Uniform fluid approximation}. We first need to prove that the shifted workload and shifted queue length are bounded at zero, which we do in the following section. Using these bounds and some properties of the cumulative service time \eqref{c}, we can apply the results in \cite[Sections 4 and 5]{bramson1998state}.


\subsection{Bounding the shifted processes}\label{Bounding the shifted processes}
First, we find the relation between the diffusion scaling and the shifted scaling. Although this relation is easily obtained and is  already known in the literature (e.g. \citep{zhang2011diffusion}), we provide it here for completeness.
Fix $L>1$ and define the shifted fluid processes on $[0,L]$. The interval $[0,n^2T]$ can be covered by $[nt]+1$ overlapping intervals as follows. For $t\in[0,n^2T],$ there exist $s\in[0,L]$ and $m\in\{0,1,\ldots,[nt]\}$ such that
\begin{equation*}
n^2t=nm+ns.
\end{equation*}
We can write the relation between the diffusion scaling and the shifted scaling as follows. For $s\leq L,$
\begin{align}
\bar{Q}^{n,m}(s)&=\wh{Q}^n\big(\frac{nm+ns}{n^2}\big),\label{T1a}\\
\bar{W}^{n,m}_{Tot}(s)&=\wh{W}^n_{Tot}\big(\frac{nm+ns}{n^2}\big)\label{T1b}.\\ \nonumber
\end{align}

By Lemma~\ref{l1} and \eqref{T1b}, it follows that for any $\epsilon>0$ there exists a constant $B_1>0$ such that
\begin{equation}\label{eq:Bound for workload}
\liminf_{n\rightarrow \infty}
\Prob(\max_{m\leq nT}\bar{W}^{n,m}_{Tot}(0)\leq B_1)\geq1-\epsilon.
\end{equation}
We denote the event $\{\omega\in\Omega: \max_{m\leq nT}\bar{W}^{n,m}_{Tot}(0)\leq B_1\}$ by $\mathcal{G}_1^n(B_1).$ Using now \eqref{eq:Bound for workload}, it can be shown that the shifted queue length process is stochastically bounded at zero.
\begin{lem}\label{l2}
Let $T>0.$ For any $\epsilon>0$ there exists a constant $B_2>0$ such that \begin{equation}\label{eq:Bound for queue}
\liminf_{n\rightarrow \infty}
\Prob(\max_{m\leq nT}
\bar{Q}^{n,m}_{i}(0)\leq B_2)
\geq1-\epsilon, {\hskip 5pt}i=1,\I.
\end{equation}
We denote the event $\{\omega\in\Omega: \max_{m\leq nT}
\bar{Q}^{n,m}_{i}(0)\leq B_2\}$ by $\mathcal{G}^n_2(B_2)$.
\end{lem}
\begin{proof}
We prove the result by deriving a contradiction, so suppose that \eqref{eq:Bound for queue} does not hold. Thus, assume there exists at least one $i$ such that  $Q_i^{n,m}(0)$ is not stochastically bounded. In other words, there exists a $\delta >0$ such that for any $B>0$,
\begin{equation}\label{100}
\liminf_{n\rightarrow \infty}
\Prob\big( \max_{m\leq nT}
Q_i^{n}(nm)>Bn \big)>\delta.
\end{equation}
Suppose that $m_n$ is such that it optimises the quantity
$\max_{m\leq nT}Q_i^{n}(nm)$. We can choose in \eqref{eq:Bound for workload}, $\epsilon=\frac{\delta}{3}$ and a (large enough) constant $B_1$. Also, we choose a constant $B$ such that $B>2\frac{B_1}{\beta_i}$.
By the definition of the total workload \eqref{eq:total_workload} and \eqref{eq:HL}, we have that
\begin{align*}
\max_{m\leq nT} W^{n}_{Tot}(nm)&\geq\sum_{j=1}^{S_i^{n}(T^n_i(nm_n))+1}v_{ii}^{(1)}(j)-T^n_i(nm_n)
+\sum_{j=S_i^{n}(T^n_i(nm_n))+2}^{S_i^{n}(T^n_i(nm_n))+Q_i^{n}(nm_n)}v_{ii}^{(1)}(j)\\
&\geq \sum_{j=S_i^{n}(T^n_i(nm_n))+2}^{S_i^{n}(T^n_i(nm_n))+Q_i^{n}(nm_n)}v_{ii}^{(1)}(j).
\end{align*}
We know that $v_{ii}^{(1)}(j)$ are i.i.d.\ with mean $\beta_i$. Also, in the previous summation $j>S_i^{n}(T^n_i(nm_n))$, which means that $v_{ii}^{(1)}(j)$ are independent of the process $S_i^{n}(T^n_i(nm_n))$. 
Define the the following event
\begin{equation}\label{extraset}
\mathcal{G}^n=\{\omega\in\Omega:\big|\frac{1}{Bn}\sum_{j=S_i^{n}(T^n_i(nm_n))+2}^{S_i^{n}(T^n_i(nm_n))+Bn}
v_{ii}^{(1)}(j)-\beta_i\big|<\frac{\beta_i}{2}\}.
\end{equation}
By the weak law of large numbers (which we can apply due to the independence of the $v_{ii}^{(1)}(j)$ and $S_i^{n}(T^n_i(nm_n))$), we have that for large $n$, 
$\Prob(\mathcal{G}^n)\geq 1-\frac{\delta}{3}$. In the sequel, we assume that 
$\omega \in \mathcal{G}^n\cap\mathcal{G}_1^n(B_1)\cap (\mathcal{G}_2^n(B))^c$, and note that
$\Prob{(\mathcal{G}^n\cap\mathcal{G}_1^n(B_1)\cap (\mathcal{G}_2^n(B))^c)}\geq \frac{\delta}{3}$.
Applying \eqref{100} and dividing by $n$ we derive
\begin{equation*}\label{ql}
B_1\geq \max_{m\leq nT}\bar{W}^{n,m}_{Tot}(0)\geq
\frac{B}{Bn}
\sum_{j=S_i^{n}(T^n_i(nm_n))+2}^{S_i^{n}(T^n_i(nm_n))+Bn}
v_{ii}^{(1)}(j).
\end{equation*}
By \eqref{extraset} and the last inequality, we obtain for sufficiently large $n$,
$B_1\geq B(\beta_i-\frac{\beta_i}{2})>2\frac{B_1}{2}$. This yields  a contradiction.
\end{proof}	
Having proved that the shifted processes are bounded, we can show that the shifted processes can be approximated by a solution of the fluid model. This is the topic of next section, in which we use a very similar approach as is Bramson \citep[Sections 4 and 5]{bramson1998state}.


\subsection{Uniform fluid approximation}\label{Uniform fluid approximation}
By \cite[Proposition~5.1]{bramson1998state}, we have that for any $\epsilon>0$,
\begin{equation}\label{L1}
\Prob(\max_{m< nT} \|\bar{E}^{n,m}(\cdot)-\lnn \cdot\|_L>\epsilon)\leq\epsilon.
\end{equation}
Also, by \cite[Proposition~5.2]{bramson1998state} it is known that the shifted arrival process is almost Lipschitz continuous, which means that for some $N_1>0,$
\begin{equation*}
\Prob (\sup_{t_1,t_2\in[0,L]}
{|\bar{ {E}}^{n,m}(t_2)-\bar{ {E}}^{n,m}(t_1)|
>N_1|t_2-t_1|}+\epsilon\
for\ some\ m<nT)\le\epsilon.
\end{equation*}
Furthermore, using the definition of the cumulative service time \eqref{c}, the property $\sum_{i=1}^{\I}  R_i^n(q)= 1$, and the observation that $Y^n_{L_2}(\cdot)$ and $T^n(\cdot)$ are increasing functions in time, we conclude that the shifted process, $T^{n,m}(\cdot),$ and the shifted idle time are Lipschitz continuous with constant equal to 1.

\begin{prop}\label{p2}
Let $\epsilon>0$. Then, for an appropriate large $n$, and for $i=1,\I$,
\begin{align}
& \Prob(\max_{m< nT} \|\ba{S}_i^{n,m}(\ba{T}_i^{n,m}(\cdot))
-\mu_i \ba{T}_i^{n,m}(\cdot)\|_L>\epsilon)<\epsilon, \label{2a}\\
&\Prob(\max_{m< nT} \|\sum_{l=1}^{\I}\bar{\Phi}_{li}^{n,m}\Big(\ba{S}^{n,m}_l\big(\ba{T}^{n,m}_l(\cdot)\big)\Big)- \sum_{l=1}^{\I}\mu_lp_{li}\bar{T}_l^{n,m}(\cdot)\|_L>\epsilon)<\epsilon, \label{2c}
	\end{align}	
\begin{equation}\label{2b}
\begin{split}
	&\Prob (\max_{m< nT} \|\bar{V}^{n,m}_i\bigg(\ba{Q}^{n,m}_i(0)+\ba{E}^{n,m}_i(\cdot)+
	\sum_{l=1}^{\I}\ba{\Phi}_{li}^{n,m}\Big(\ba{S}^{n,m}_l\big(\ba{T}^{n,m}_l(\cdot)\big)\Big)\bigg)\\
	&\qquad\qquad \qquad-
\beta_i(\ba{Q}_i^{n,m}(0)+\lin \cdot+
\sum_{l=1}^{\I}p_{li}\bar{T}_l^{n,m}(\cdot))\|_L
>\epsilon)<\epsilon,\\
\end{split}
\end{equation}
\end{prop}
\begin{proof}
It is shown in \cite{bramson1998state} that for a renewal process
$S(\cdot)$,
\begin{equation*}
\Prob\big(\sup_{m\leq nT}
\sup_{t \leq L}|\frac1n(S_i^n((nm+nt))
-S_i^n(nm))-\mu t|_L\geq\epsilon)
<\epsilon,
\end{equation*}
which is equivalent to (the process can start anywhere in the interval $[0,n^2T]$)
\begin{equation}\label{bram}
\Prob\big(\sup_{u\in[0,n^2T]}
\sup_{t \leq L}
|\frac1n(S_i^n((u+nt))-S_i^n( u))
-\mu_i t|\geq\epsilon)
<\epsilon.
\end{equation}
Let $t'=\ba{T}_i^{n,m}(t)=\frac{1}{n}(T_i^n(nm+nt)-T_i^n(nm))\in[0,L]$ and $u=T_i^n(nm)\leq n^2T$, for $m\leq nT$. By \eqref{bram} we obtain
\begin{equation*}\label{our}
\Prob\big(\sup_{u\in[0,n^2T]}
\sup_{t' \leq L}
|\frac1n(S_i^n((u+nt'))
-S_i^n(u))-\mu_it'|\geq\epsilon\big)
<\epsilon,
\end{equation*}
for each $i=1,\I$. Then, \eqref{2a} follows.

By the Lipschitz continuity of the departure process in Proposition~\ref{p3}, which we can prove only using \eqref{2a}, we know that $\|S_i(T_i(\cdot))\|_L\leq N_2Ln$. Using \cite[Proposition~4.2]{bramson1998state}, we derive
\begin{equation*}
\Prob(\max_{m< nT}
 \|\sum_{l=1}^{\I}
 \bar{\Phi}_{li}^{n,0}
 \Big(\ba{S}^{n,0}_l\big(\ba{T}^{n,0}_l(\cdot)\big)\Big)
 - \sum_{l=1}^{\I}
 p_{li}\bar{S}_l^{n,0}(\ba{T}^{n,0}(\cdot))\|_L
 >N_2L\epsilon)
 <\frac\epsilon n.
\end{equation*}
Furthermore, using the conclusion of the proof of \cite[Proposition~5.19]{bramson1998state}, we obtain
\begin{equation*}
\Prob(\max_{m< nT} \|\sum_{l=1}^{\I}\bar{\Phi}_{li}^{n,m}\Big(\ba{S}^{n,m}_l\big(\ba{T}^{n,m}_l(\cdot)\big)\Big)- \sum_{l=1}^{\I}p_{li}\bar{S}_l^{n,m}(\ba{T}^{n,m}(\cdot))\|_L>N_2L\epsilon)<\frac\epsilon n.
\end{equation*}
Applying \eqref{2a} to the last inequality we obtain \eqref{2c}.
	
To prove \eqref{2b}, we know by Lemma~\ref{l2}, that for some $B_2>0$,
$
|Q^n(0)|\leq B_2n.
$
 Using \eqref{L1}, \eqref{2b}, and applying \cite[Proposition~4.2]{bramson1998state}, the result follows.
\end{proof}

In the following proposition, we show that all the shifted processes are almost Lipschitz continuous.
\begin{prop}\label{p3}
Let $\bar{ {X}}^{n,m}(\cdot)$ be any of the processes
$\bar{S}^{n,m}(\cdot),$ $\bar{Q}^{n,m}(\cdot)$ and $\bar{W}^{n,m}(\cdot).$ Then for large $n,$ for $\eps>0$ and some $N>0,$ we have that
\begin{equation}\label{1}
\Prob (\sup_{t_1,t_2\in[0,L]}
 {|\bar{ {X}}^{n,m}(t_2)
 -\bar{ {X}}^{n,m}(t_1)|>N|t_2-t_1|}+\epsilon\
\textrm{for some}\ m<nT)\le\epsilon.
\end{equation}
\end{prop}
\begin{proof}
For the departure process and by using \eqref{2a}, we have for  $i=1,\I$ that
\begin{align*}
|\bar{S_i}^{n,m}(T_i^{n,m}(t_2))
-\bar{S_i}^{n,m}(\ba{T}_i^{n,m}(t_1))|\leq |\bar{S_i}^{n,m}&(\ba{T}_i^{n,m}(t_2))
-\mu_i\ba{T}_i^{n,m}(t_2)|+|\bar{S_i}^{n,m}(\ba{T}_i^{n,m}(t_1))
-\mu_i\ba{T}_i^{n,m}(t_1)|\\
&+|\mu_i\ba{T}_i^{n,m}(t_2)
-\mu_i\ba{T}_i^{n,m}(t_1)|
\leq N_2|t_2-t_1|+2\epsilon,
\end{align*}
where $N_2=\max_i \mu_i$.
Using \eqref{L1}, \eqref{2b}, and the
Lipschitz continuity of the cumulative service time \eqref{c},
$T(\cdot),$ it is easy to show that the shifted total arrival process, $\ba{A}^{n,m}(\cdot),$ is almost Lipschitz continuous with $N_3=N_1+\|P\|N_2$, where $N_1=\max_i\lambda_i$.

Combining the almost Lipschitz continuity for the shifted arrival and the shifted departure process, $\ba{A}^{n,m}(\cdot),$ $\ba{S}^{n,m}(T(\cdot)),$ the result for the shifted queue length process, $\ba{Q}^{n,m}(\cdot),$ follows with the constant $N_4=N_3+N_2$. Using the same idea and \eqref{2b}, we obtain the same result for the shifted immediate workload process, $\ba{W}^{n,m}(\cdot),$ with
$N_5=\frac{N_1}{\min_i{\mu_i}}+\|P\|.$
\end{proof}

\begin{rem}\label{r1}
{
Adapting the techniques in \cite{bramson1998state} we can replace $\epsilon$ in the propositions above
by $\epsilon(n)$ such that $\epsilon(n)\rightarrow 0.$
Let $\mathcal{G}_i^n \subseteq\Omega$, $1\leq i \leq 5$, be the ``good events'' such that the complements of inequalities \eqref{L1}, \eqref{2a}, \eqref{2c}, \eqref{2b} and \eqref{1} hold if we replace $\epsilon$ by $\epsilon(n)$. Also, let
$\mathcal{G}_1^n(B)$ and $\mathcal{G}_2^n(B)$ be as in \eqref{eq:Bound for workload} and \eqref{eq:Bound for queue} with $B=\max\{B_1,B_2\}.$ Denote by $\mathcal{G}_0^n(B)$ the intersection of the previous events.
}
Because $\omega\in \mathcal{G}_0^n(B)$, we know that
$|\bar{\qnp}^{n,m}(t_2)
-\bar{\qnp}^{n,m}(t_1)|
\leq N|t_2-t_1|+\epsilon.
$
Also, by Lemmas~\ref{l1}, \ref{l2}, and by the definition of the shifted processes \eqref{RWa}$\textup{--}$\eqref{RWd}, we have that
$|\bar{\qnp}^{n,m}(0)|\leq B,$
for some positive constant B. In addition, if we replace the bound in \cite[Inequality~4.6]{bramson1998state}, by a general real number, we can again show that the set of Lipschitz functions with this property is compact; see \cite[Lemma 6.3]{ye2005stability}.
\end{rem}
\noindent
By Remark~\ref{r1}, all the requirements in \cite[Section~4.1]{bramson1998state} hold. Thus, we can find a Lipschitz-continuous function $\wt{\qnp}(\cdot)$, such that for $\epsilon(n)\to 0,$
\begin{equation}\label{eq:fluid approximation}
\|\bar{\qnp}^{n,m}(\cdot,\omega)-\wt{\qnp}(\cdot)\|_L\le \epsilon(n), {\hskip 15pt} \forall\ \omega\in\mathcal  G^n_0(B),\ \forall\ m\le nT.
\end{equation}
\begin{prop}\label{p5}
The function $\wt{\qnp}(\cdot)$ is a solution to the fluid model equations on $[0,L]$.
\end{prop}
\begin{proof}
We shall show that $\wt{\qnp}(\cdot)$ verifies the fluid model's equations
\eqref{eq:Fluid queue}$\textup{--}$\eqref{eq:Fluid idle time}. To do this, let
$\delta>0$. As $\epsilon(n)\rightarrow 0,$ we can find large a $n$ such that $\epsilon(n)<\delta.$ It is known by \eqref{eq:fluid approximation} that for large $n,$
\begin{equation*}
\|\bar{\qnp}^{n,m}(\cdot)-\wt{\qnp}(\cdot)\|_L<\delta.
\end{equation*}
Thus, from the heavy traffic assumption, we conclude that $|\lnn-\lambda|<\delta.$
Using the above inequalities, \eqref{L1}, Proposition~\ref{p2}, and the triangle inequality it can be proved in the same way as in \citep[Proposition~6.2]{bramson1998state} that all the functions $\wt{\qnp}(\cdot),$ except for $\wt{T}(\cdot),$ verify the fluid model's equations.
To prove that $\wt{T}(\cdot)$ satisfies \eqref{eq:Fluid busy time}, we need to use the following two properties of the service allocation function: i) $R^{n}(nq)=R(q)$ and ii) $R(\cdot)$ is a Lipschitz continuous function on
$\mathbb R^\I_+\backslash\{0\}$; i.e., there exists a constant $C,$ such that for $q_1, q_2\in \mathbb{R}^\I_+\backslash\{0\}$
\begin{equation*}\label{4.6}
|R(q_2)-R(q_1)|\leq C|q_2-q_1|.
\end{equation*}
Now, using \eqref{RWc}, and the above properties of the service allocation function, we can show that $\wt{T}(\cdot)$ satisfies \eqref{eq:Fluid busy time} and is thus a solution to the fluid model:
\begin{equation*}
\begin{split}
|\wt{T}(t)-\int_{0}^{t} R(\wt{Q}(s))ds|
&= |\wt{T}(t)-\bar{T}^{n,m}(t)
+\int_{0}^{t} R(\bar{Q}^{n,m}(s))ds
-\int_{0}^{t} R(\wt{Q}(s))ds| \\
&\leq |\wt{T}(t)-\bar{T}^{n,m}(t)|
+ \int_{0}^{t} |R(\bar{Q}^{n,m}(s))
-R(\wt{Q}(s))|ds \\
&\leq \delta+
\int_{0}^{t} C|\bar{Q}^{n,m}(s)-\wt{Q}(s)|
\leq \delta+ \int_{0}^{t} C\delta
\leq \delta+  CL \delta.
\end{split}
\end{equation*}

\end{proof}

\subsection{The scaled shifted total workload process}\label{The scaled shifted total workload process}

In this section, we see that we can approximate the scaled shifted total workload process by a solution to the fluid model. We begin with a preliminary result.

\begin{prop}\label{extra1}
For appropriately large $n\in \mathbb{N}$ and $\epsilon>0,$ we have that
\begin{equation}\label{eq:extra1}
\begin{split}
&\Prob \Big(\max_{m< nT} \Big \|
\sum_{i=1}^{\I}\frac1n\sum_{j=Z^{n,m}_i(\cdot)}^{B^{n,m}_i(\cdot)}
s'_i(j)-
\sum_{k=1}^{\infty}\beta^T(P^T)^{k}
\ba{Q}^{n,m}(\cdot)\Big \|_L
>\epsilon\Big)<\epsilon,
\end{split}
\end{equation}
where
\begin{equation*}
\begin{split}
Z^{n,m}_i(\cdot)&=n\ba{S}^{n,m}_i
\big( \ba{T}^{n,m}_i(\cdot)\big)
+S^{n}_i\big( T^{n}_i(nm)\big)+1,\\
B^{n,m}_i(\cdot)&=Q^{n}_i(0)
+n\ba{A}^{n,m}_i( \cdot)+A^{n}_i(nm).
\end{split}
\end{equation*}
\end{prop}
\begin{proof}
The random variables $s'_i(j)$
depend on $n$, but to keep the notation simple we omit the index $n$.
Note that \eqref{eq:extra1} can be written as
\begin{equation*}
\begin{split}
&\Prob \Big(\max_{m< nT} \Big \|
\sum_{i=1}^{\I}\frac1n\sum_{j=Z^{n,m}_i(\cdot)}^{B^{n,m}_i(\cdot)}
s'_i(j)-
\sum_{k=1}^{\infty}\beta^T\ba{p}_i^{(k)}
\ba{Q}_i^{n,m}(\cdot)\Big \|_L
>\epsilon\Big)<\epsilon,\\
\end{split}
\end{equation*}
where
$\beta^T\ba{p}^{(k)}_i=\beta_1p^{(k)}_{i1}+\beta_2p^{(k)}_{i2}.$
We have that
\begin{equation*}
\begin{split}
&\Prob \Big(\max_{m< nT}\Big \|
\sum_{i=1}^{\I}\frac1n\sum_{j=Z^{n,m}_i(\cdot)}^{B^{n,m}_i(\cdot)}
s'_i(j)-
\sum_{k=1}^{\infty}\beta^T\ba{p}_i^{(k)}
\ba{Q}_i^{n,m}(\cdot)\Big\|_L
>\epsilon\Big)\\
&\quad\leq \sum_{i=1}^{\I} \Prob \Big(\max_{m< nT}\Big \|
\frac1n\sum_{j=Z^{n,m}_i(\cdot)}^{B^{n,m}_i(\cdot)}
s'_i(j)-
\sum_{k=1}^{\infty}\beta^T\ba{p}_i^{(k)}
\ba{Q}_i^{n,m}(\cdot)\Big\|_L
>\frac{\epsilon}{2}\Big),
\end{split}
\end{equation*}
so it is enough to show that the last term is sufficiently small for $i=1,2$. First, we shall prove it for $m=0.$
We know  by Proposition~\ref{p3} that the shifted queue length process is Lipschitz continuous and by Proposition~\ref{l2} that the shifted queue length process at zero is stochastically bounded; i.e., for $t\leq L$, $|Q_i^n(t)|\leq(N_4+B_2)Ln$. By \citep[Proposition 4.2]{bramson1998state}, we derive
\begin{equation*}
\begin{split}
\Prob \Big(\Big \|
\sum_{j=S^{n}_i\big( T^{n}_i(\cdot)\big)}^
{Q^n_i(0)+A^{n}_i(\cdot)}
s'_i(j)-
\sum_{k=1}^{\infty}\beta^T\ba{p}_i^{(k)}
Q_i^n(\cdot)\Big\|_{N_4Ln}
>\frac{\epsilon (N_4+B_2)Ln}{2}\Big)\leq \frac{\epsilon}{2(N_4+B_2)Ln},
\end{split}
\end{equation*}
which leads to
\begin{equation*}
\begin{split}
\Prob \Big(\Big\|\frac{1}{n}
\sum_{j=Z^{n,0}_i(\cdot)}^
{B^{n,0}_i(\cdot)}
s'_i(j)-
\sum_{k=1}^{\infty}\beta^T\ba{p}_i^{(k)}
\ba{Q}_i^{n,0}(\cdot)\Big\|_{L}
>\frac{\epsilon (N_4+B_2)L}{2}\Big)
\leq \frac{\epsilon}{2(N_4+B_2)n},
\end{split}
\end{equation*}
where $B^{n,0}_i(\cdot)-Z^{n,0}_i(\cdot)= \ba{Q}_i^{n,0}(\cdot).$
Multipling the error bounds by the number of processes $[nT]+1$ and choosing a suitable $\epsilon$ we derive
\begin{equation*}
\begin{split}
\Prob \Big(\Big\|\frac{1}{n}
\sum_{j=Z^{n,0}_i(\cdot)}^
{B^{n,m}_i(\cdot)}
s'_i(j)-
\sum_{k=1}^{\infty}\beta^T\ba{p}_i^{(k)}
\ba{Q}_i^{n,m}(\cdot)\Big\|_{L}
>\frac{\epsilon}{2}\Big)\leq \frac{\epsilon}{2}.
\end{split}
\end{equation*}
\end{proof}
Adapting Remark~\ref{r1}, we can replace $\epsilon$ in the proposition above
$\epsilon$ by $\epsilon(n)$ such that $\epsilon(n)\rightarrow 0$ as $n\rightarrow\infty.$ This will be done in the next result, where we combine all technical estimates so far to construct a ``good'' event.

\begin{prop}\label{prop:good sets}
Let $\epsilon >0$ and $\mathcal{G}_0^n \subseteq\Omega$ be as in Remark~\ref{r1}. Let $\mathcal{G}_6^n\subseteq\Omega$ be the event
such that the complement of \eqref{eq:extra1} holds if we replace $\epsilon$ by $\epsilon(n)$. Define the event
$\mathcal{G}^n(B)=\mathcal{G}_0^n(B) \cap \mathcal{G}_6^n.$
Then
\begin{equation*}
\lim_{n\rightarrow\infty}\Prob{\big(\mathcal{G}^n(B)\big) }
\geq 1-\epsilon.
\end{equation*}
\end{prop}

\begin{proof}
Note that
$\mathcal{G}^n(B)=\bigcap_{i=1}^{6} \mathcal{G}_i^n \cap \mathcal{G}_1^n(B)\cap \mathcal{G}_2^n(B).$
We denote $(\mathcal{G}_i)^c$ the complement of the event $\mathcal{G}_i.$ For $1\leq i \leq 6$, we have by construction that
$
\Prob{((\mathcal{G}_i^n)^c)}\leq \epsilon(n).
$
Also, by \eqref{eq:Bound for workload} and \eqref{eq:Bound for queue}
we can choose a constant $B$ such that for $i=1,\I$
\begin{equation*}
\lim_{n\rightarrow\infty}\Prob{
\big(\mathcal{G}_i^n(B) \big)}\geq 1-\frac{\epsilon}{2}.
\end{equation*}
Combining the above inequalities, we obtain as $n\rightarrow\infty,$
\begin{equation*}
\begin{split}
\Prob{(\mathcal{G}^n(B))}=
\Prob{\Big(\Big(\bigcup_{i=1}^{6}
(\mathcal{G}_i^n)^c
\cup (\mathcal{G}_1^n(B))^c
\cup (\mathcal{G}_2^n(B))^c \Big)^c\Big)}
=&1-\Prob{\Big(\bigcup_{i=1}^{6}(\mathcal{G}_i^n)^c
\cup (\mathcal{G}_1^n(B))^c
\cup (\mathcal{G}_2^n(B))^c
\Big)}\\
\geq& 1-6\epsilon(n)-\epsilon \rightarrow 1 -\epsilon.
\end{split}
\end{equation*}
\end{proof}

In the sequel, we assume that $\omega\in\mathcal{G}^n(B).$ In other words, Proposition~\ref{prop:good sets} allows us to use a sample-path approach.
As a final step towards proving state-space collapse, we use this approach to show that there exists a fluid approximation for the total workload of the system.

\begin{prop}\label{Prop:extra}
There exists a solution of the fluid model equations, $\wt{W}_{Tot}(\cdot)$, such that for $m\leq nT$

\begin{equation*}\label{In:approximation of workload}
\|\bar{W}^{n,m}_{Tot}(\cdot,\omega)-\wt{W}_{Tot}(\cdot)\|_L\le
 \epsilon(n)
{\hskip 15pt} \forall\ \omega\in \mathcal{G}^n(B).
\end{equation*}
\end{prop}

\begin{proof}
Take $\epsilon >0$ and let $\omega\in\mathcal{G}^n(B)$ as defined in Proposition~\ref{prop:good sets}.
Let the functions $\wt Q(t)$ and  $\wt W(t)$ be such so that they satisfy \eqref{eq:fluid approximation}. Define the function
\begin{equation*}\label{eq:Sofmw}
\wt{W}_{Tot}(t)=\beta^T(I-P^T)^{-1}\wt Q(t),
\end{equation*}
which is a solution to the fluid model equations because $\wt Q(t)$ is.
Omitting again the index $n$ in the quantity $s'_{i}(j),$ by the definition of the total workload, we have that
\begin{equation*}\label{eq:ShTW}
\ba{W}^{n,m}_{Tot}(t)=\sum_{i=1}^{2}\ba{W}^{n,m}_i(t)+
\sum_{i=1}^{\I}
\frac1n\sum_{j=Z_i^{n,m}(t)}
^{B_i^{n,m}(t)}
s'_i(j),
\end{equation*}
where the quantities $Z_i^{n,m}(t)$ and $B_i^{n,m}(t)$ are defined in Proposition~\ref{extra1}. Using the triangular inequality, we thus have that
\begin{equation*}
\begin{split}
\Big\|\bar{W}_{Tot}^{n,m}(\cdot)-\wt{W}_{Tot}(\cdot)\Big\|_L\leq&
\sum_{i=1}^{\I}\Big\|\bar{W}_i^{n,m}(\cdot)-\wt{W}_i(\cdot)\Big\|_L
+\Big\|
\sum_{i=1}^{\I}\frac1n\sum_{j=Z_i(t)}^{B_i(\cdot)}
s'_i(j)-\sum_{k=1}^{\infty}\beta^T(P^T)^{k} \ba{Q}^{n,m}(\cdot)\Big\|_L\\
&\quad +\Big\|\sum_{k=1}^{\infty}\beta^T(P^T)^{k}
(\ba{Q}^{n,m}(\cdot)-\wt Q(\cdot))\Big\|_L.
\end{split}
\end{equation*}
By \eqref{eq:fluid approximation} and \eqref{eq:extra1},
$
\Big\|\bar{W}_{Tot}^{n,m}(\cdot)-\wt{W}_{Tot}(\cdot)\Big\|_L
\leq
(2
+1
+\max_{i}{(\tau_i-\beta_i)})\epsilon(n).
$
\end{proof}

\subsection{State-space collapse}\label{State-space collapse}

Now, we can state and prove SSC for the diffusion queue length process.
\begin{thm}[SSC]\label{t2} Assume that
	\begin{equation}\label{as}
	|\wh{Q}^n(0)-\Delta\wh{W}_{Tot}^n(0)|
	\to 0 {\hskip 15 pt}\textrm{in probability,}
	\end{equation}
as $n\to \infty$. Then for any $T>0$,
	\begin{equation}\label{6}
	\|\wh{Q}^n(\cdot)-\Delta\wh{W}_{Tot}^n(\cdot)\|_{T}
	\to 0 {\hskip 15 pt}\textrm{in probability,}
	\end{equation}
	as $n\to \infty$.
\end{thm}
\begin{proof}
Take $\epsilon >0$ and let $\omega\in\mathcal{G}^n(B)$ as defined in Proposition~\ref{prop:good sets}.
 By Theorem~\ref{Thm: Convergence of the fluid model}, we know that there exists a constant $L^{\ast}$,
such that for $t\geq L^{\ast},$
\begin{equation}\label{t21}
|\wt{Q}(t)-\Delta \wt{W}_{Tot}(t)|\le \epsilon.
\end{equation}
Fix $L>L^{\ast}+1$. It is known that
\begin{equation*}
[0,n^2T]\subseteq [0,nL^{\ast}] \bigcup_{m=0}^{[nT]}
 [n(m+L^{\ast}),n(m+L)].
\end{equation*}
So, it suffices to show that
\begin{equation}\label{t22}
\max_{m\leq nT}\sup_{t\in[L^{\ast},L]}|\bar{Q}^{n,m}(t)-\Delta
\bar{W}_{Tot}^{n,m}(t)|<\epsilon
\end{equation}
and
\begin{equation}\label{t23}
\sup_{t\in[0,L^{\ast}]}|\bar{Q}^{n,0}(t)-\Delta \bar{W}_{Tot}^{n,0}(t)|<\epsilon.
\end{equation}
Then, by using \eqref{T1a} and \eqref{T1b}, we derive \eqref{6}.
To prove \eqref{t22}, we know that by \eqref{eq:fluid approximation} and Proposition~\ref{Prop:extra}, for $t\leq L,$
\begin{equation}\label{t24}
|\bar{Q}^{n,m}(t)-\wt{Q}(t)|<\epsilon
\end{equation}
and \begin{equation}\label{t25}
|\bar{W}_{Tot}^{n,m}(t)-\wt{W}_{Tot}(t)|<\epsilon.
\end{equation}
Recall that the lifting map is Lipschitz continuous with constant
$C_1$. Combining this with \eqref{t21}, \eqref{t24}, and \eqref{t25}, we get \eqref{t22}.

To prove \eqref{t23}, we have
by \eqref{eq:fluid approximation} that for $t\leq L$,
\begin{center}
$
|\bar{Q}^{n,0}(t)-\wt{Q}(t)|<\epsilon
$
and
$
|\bar{W}_{Tot}^{n,0}(t)-\wt{W}_{Tot}(t)|<\epsilon.
$
\end{center}
Also, by Assumption~\eqref{as} we obtain
$
|\bar{Q}^{n,0}(0)-\Delta \bar{W}_{Tot}^{n,0}(0)|<\epsilon.
$
By the last three inequalities and Proposition~\ref{starting on IM},
we can apply \citep[Lemma~6.1]{bramson1998state} and derive
$
|\wt{Q}(t)-\Delta \wt{W}_{Tot}(t)|<\epsilon
$
for $0 \leq t\leq L^{\ast}.$ In a similar way as before, we get \eqref{t23}.
\end{proof}
Now, we are ready to prove Theorem~\ref{Thm:main result}, which is a result of Lemma~\ref{l1}, Theorem~\ref{t2}, and the continuous mapping theorem.
\begin{proof}[Proof of Theorem~\ref{Thm:main result}]
By Lemma~\ref{l1}, we know that $\wh{W}^n(\cdot) \overset{d} \rightarrow \wh{W}_{Tot}(\cdot).$ Also, the lifting map,
$\Delta$ is continuous. Applying the continuous mapping theorem \citep[Theorem~1.2]{durrett1996stochastic}, we have that
$\wh{W}^n(\cdot) \overset{d} \rightarrow \Delta\wh{W}_{Tot}(\cdot).$  Now, the result follows by Theorem~\ref{t2} and the
converging together lemma in \citep[Lemma~1.3]{durrett1996stochastic}.
\end{proof}

\phantomsection
\addcontentsline{toc}{section}{Acknowledgements}
\section*{Acknowledgements}
\begin{footnotesize}
This work was done in part while the authors were visiting the Simons Institute for the Theory of Computing, Berkeley.
The research of Angelos Aveklouris is funded by a TOP grant
of the Netherlands Organization for Scientific Research (NWO) through project 613.001.301.
The research of Maria Vlasiou and Jiheng Zhang is partly supported by
two grants from the `Joint Research Scheme' program, sponsored by NWO and the Research
Grants Council of Hong Kong (RGC) through projects 649.000.005 and DHK007/
11T, respectively. The research of Bert Zwart is partly supported by
an NWO VICI grant.\end{footnotesize}

\begin{footnotesize}The authors would like to wish Professor Tomasz Rolski many more happy, productive and inspirational years. We are grateful for his guidance, friendship, and insights throughout our careers. The legacy he has build has had a profound impact on our academic lives. It has enriched our collaborations, formed friendships, and made us broader as researchers and human beings.
\end{footnotesize}

\cleardoublepage\phantomsection
\addcontentsline{toc}{section}{References}
\begin{footnotesize}
\bibliography{biblio}
\bibliographystyle{abbrv}
\end{footnotesize}
\end{document}